\documentclass[12pt]{amsart}
\usepackage{amsmath}
\usepackage{amsthm}
\usepackage{amssymb}
\usepackage{amscd}
\usepackage{bbold}
\usepackage{graphicx}
\usepackage{enumitem}
\usepackage{subfigure}
\usepackage{cite}
\usepackage{lscape} 
\usepackage{indentfirst}
\usepackage{multirow}
\usepackage{tabls}
\usepackage{wrapfig}
\usepackage{slashbox}
\usepackage{longtable}
\usepackage{supertabular} 

\newtheorem{lem}{Lemma}[section]
\newtheorem{thm}[lem]{Theorem}
\newtheorem{prop}[lem]{Proposition}
\newtheorem{cor}[lem]{Corollary}

\theoremstyle{definition}
\newtheorem{defn}[lem]{Definition}
\newtheorem{example}[lem]{Example}

\newtheorem{rem}[lem]{Remark}

\newcommand{\CC}{{\mathbb C}}

\newcommand{\FF}{{\mathbb F}}

\newcommand{\NN}{{\mathbb N}}

\newcommand{\RR}{{\mathbb R}}

\newcommand{\ZZ}{{\mathbb Z}}

\newcommand{\Acal}{\mathcal{A}}
\newcommand{\Bcal}{\mathcal{B}}
\newcommand{\Ccal}{\mathcal{C}}
\newcommand{\Dcal}{\mathcal{D}}

\newcommand{\Fcal}{\mathcal{F}}

\newcommand{\Lcal}{\mathcal{L}}

\newcommand{\Scal}{\mathcal{S}}

\newcommand{\Ucal}{\mathcal{U}}

\def\mbbu{\mathbb{1}}                   
\def\benm{\begin{enumerate}}            
\def\ennm{\end{enumerate}}              
\newcommand{\norm}[1]{\left\Vert #1\right\Vert}         
\newcommand{\inner}[2]{\langle {#1,#2}\rangle}

\title{ Frame multiplication theory 
and  a vector-valued DFT and ambiguity function  }

\date{\today}

\author{Travis D. Andrews}

\author{John J. Benedetto}
\address{Norbert Wiener Center\\
         Department of Mathematics \\
         University of Maryland \\
         College Park, MD 20742 \\
         USA}
\email{jjb@math.umd.edu}
\urladdr{http://www.math.umd.edu/\textasciitilde jjb}

\author{Jeffrey J. Donatelli}



 \thanks{The first named author gratefully acknowledges the support
      of the Norbert Wiener Center. The second named author gratefully acknowledges the
      support of  DTRA Grant 1-13-1-0015 and ARO Grants W911NF-15-1-0112, 16-1-0008,
      and 17-1-0014.
      The third named author gratefully acknowledges the support of  the Norbert Wiener Center.}

\begin{document}

\newcounter{bean}

\begin{abstract}

Vector-valued discrete Fourier transforms ({\it DFTs}) and
ambiguity functions are defined. The motivation for the {\it definitions} is to
provide realistic modeling of multi-sensor environments in which a useful time-frequency 
analysis is
essential. 
The definition of the {\it DFT} requires associated {\it uncertainty principle inequalities}.
The definition of the ambiguity function requires a component that leads 
 to formulating a mathematical theory
in which two essential algebraic operations can be made compatible in a natural way. The theory is
referred to as {\it frame multiplication theory}.  These definitions, inequalities, and 
theory are interdependent, and they are the content of the paper with the centerpiece being
frame multiplication theory.

The technology underlying frame multiplication theory is the theory of frames, short time 
Fourier transforms ({\it STFTs}),
and the representation theory of finite groups. 
The main results have the following form:
frame multiplication exists if and only if the finite frames that arise 
in the theory are of a certain 
type, e.g., harmonic frames, or, more generally, group frames.


In light of the complexities and the importance of the modeling of time-varying 
and dynamical systems in the context of
effectively analyzing vector-valued multi-sensor environments, 
the theory
of vector-valued {\it DFTs} and ambiguity functions must not
only be {\it mathematically meaningful}, but it must have 
{\it constructive
implementable algorithms}, and be {\it computationally viable}. 
This paper presents our vision for resolving these issues, in terms of a 
significant mathematical theory, and based on
the goal of formulating and developing a useful vector-valued theory.

\end{abstract}


\maketitle

\section{Introduction}
\subsection{Background}

Our {\it background} for this work
was based in the following program.

\begin{itemize}
\item Originally, our {\it problem} was to construct
 libraries of phase-coded waveforms $v : \RR \longrightarrow \CC,$ parameterized by
design variables,
for use in communications and radar. A goal was to achieve diverse narrow-band ambiguity function 
behavior of $v$
by defining
new classes of
discrete quadratic phase and number theoretic
perfect autocorrelation sequences $u :\ZZ/N\ZZ \longrightarrow \CC$ with which to define $v$
and having optimal autocorrelation behavior in a way to be defined.
\item Then, a realistic more general problem was to construct vector-valued
waveforms $v$ in terms of
vector-valued sequences $u :\ZZ/N\ZZ \longrightarrow \CC^d$ having this optimal autocorrelation
behavior.
Such sequences are relevant in light of vector sensor 
capabilities and modeling, e.g., see \cite{Li-Sto2008}, \cite{StrFri2014}.
\end{itemize}

In fact, we shall define 
periodic vector-valued discrete Fourier transforms ({\it DFTs}) and narrow-band ambiguity functions.
Early-on we understood that the accompanying theory
could not just be 
a matter of using bold-faced letters to recount
existing theory, an image used by Joel Tropp for another multi-dimensional situation.
Two of us recorded our initial results on the subject
at an invited talk at Asilomar (2008), \cite{BenDon2008}, but we did not pursue it then, because there 
was the  fundamental one-dimensional {\it problem}, mentioned above in the first
bullet, that had to be resolved. Since then, we have
made appropriate progress on this one-dimensional problem,
see \cite{BenKonRan2009, BenDat2010, BenBenWoo2012}.


\subsection{Goals and short time Fourier transform (STFT) theme}
\label{sec:goals}
In 1953, P. M. Woodward \cite{wood1953a, wood1953b} defined the narrow-band radar ambiguity 
function. The narrow-band ambiguity function is a two-dimensional function of 
delay $t$ and Doppler frequency $\gamma$ that measures the correlation between a waveform $w$ 
and its Doppler distorted version. The information given by the 
narrow-band ambiguity function is important for 
practical purposes in radar. In fact, the {\it waveform design problem} is to
construct
waveforms having ``good" ambiguity function behavior in the sense of being designed to
solve real problems. 

Since we are only dealing with narrow-band ambiguity functions, we shall suppress the words ``narrow-band" 
for the remainder.

\begin{defn}[Ambiguity function]
\label{def:ambfn}
{\it a.} The {\it ambiguity function} $A(v)$ of $v \in L^2(\RR)$ is
\begin{align}
\label{eq:ambiguitydef}
A(v)(t,\gamma) &= \int_{\RR} v(s+t)\overline{v(s)}e^{-2\pi i s \gamma}\,ds \\
   & = e^{\pi i t \gamma} \int_{\RR} v\left(s + \frac{t}{2} \right) 
   \overline{v\left(s -
\frac{t}{2} \right)} e^{-2\pi i s \gamma} ds, \nonumber
\end{align}
for $(t,\gamma) \in \RR^2$.

{\it b.} We shall only be interested in the discrete version of \eqref{eq:ambiguitydef}. 
For an $N$-periodic function $u: \ZZ/N\ZZ \rightarrow \CC$ the {\it discrete periodic ambiguity function} is
\begin{equation}
\label{eq:peraf}
A_p(u)(m,n) = \frac{1}{N}\sum_{k=0}^{N-1}u(m+k)\overline{u(k)}e^{-2 \pi i kn/N},
\end{equation}
for $(m,n) \in \ZZ/N\ZZ \times \ZZ/N\ZZ$.

{\it c.} 
If $v, w \in L^2(\RR)$, the {\it cross-ambiguity function} $A(v,w)$ of $v$ and $w$ is
\begin{align}
\label{eq:crossaf}
A(v,w)(t,\gamma) &= \int_{\RR} v(s+t)\overline{w(s)}e^{-2\pi i s \gamma}\,ds\nonumber\\
&= e^{2\pi i t \gamma}\int_{\RR} v(s)\overline{w(s-t)}e^{-2\pi i s \gamma}\,ds.
\end{align}
Evidently, $A(v) = A(v,v)$, so that the ambiguity function is a special case of the cross-ambiguity function. 

{\it d.} The {\it short-time Fourier transform} (STFT) of $v$ with respect to a
{\it window function} $w \in L^2(\RR)\setminus\{0\}$ is
\begin{equation}
\label{eq:stft}
       V_w(v)(t,\gamma) = \int_{\RR} v(s)\overline{w(s-t)}e^{-2\pi i s \gamma}\,ds
\end{equation}
for $(t,\gamma) \in \RR^2,$ see \cite{groc2001} for a definitive mathematical treatment. Thus, we think
of the window $w$ as centered at $t$, and we have
\begin{equation}
\label{eq:afstft}
      A(v,w)(t,\gamma) =  e^{2 \pi i t \gamma}\,V_w(v)(t,\gamma).
\end{equation}

{\it e.} $A(v,w)$ and $V_{w}(v)$ can clearly be defined for functions $v, w $ on ${\mathbb R}^d$
and for other function spaces besides $L^2({\mathbb R}^d).$
The quantity $|V_{w}(v)|$ is the {\it spectrogram} of $v,$
that is so important in power spectrum analysis, see, e.g., \cite{wien1930}, \cite{BlaTuk1959}, 
\cite{papo1977}, \cite{prie1981}, \cite{chil1978}, \cite{mccl1982}, 
\cite{thom1982}.

\end{defn}

Our {\it goals} are the following.

\begin{itemize}
\item Ultimately, we shall establish the theory of vector-valued ambiguity functions 
of vector-valued functions $v$ on ${\mathbb R}^d$ in terms of their discrete periodic 
counterparts on $\ZZ/N\ZZ.$
\item To this end, in this paper, we define vector-valued {\it {\it DFT}s} and discrete periodic
vector-valued ambiguity functions on $\ZZ/N\ZZ$
in a natural way.
\end{itemize}

The STFT is the {\it guide} and the {\it theory of frames},
especially the theory of {\it DFT}, harmonic, and group frames, is the
framework (sic) to formulate these goals. The underlying technology that allows us
to obtain these goals is
frame multiplication theory.

\subsection{Outline}
We begin with an extended exposition on the theory of frames (Section \ref{sec:frames}).
The reason is that frames are
essential for our results, {\it and} our results are sometimes not conceived in terms of frames.
As such, it made sense to add sufficient background material.

The vector-valued discrete Fourier transform ({\it DFT}) is developed in Subsection \ref{sec:dftdefinv}.
The remaining two subsections of Section \ref{sec:vvDFT} conclude with a comparison of 
relations between Subsection \ref{sec:dftdefinv} and apparently different implications from the
Gelfand theory. Subsection \ref{sec:dftdefinv} is required in our vector-valued
ambiguity function theory.

Section \ref{sec:stftform} establishes the basic role of the STFT
in achieving the goals listed in Subsection \ref{sec:goals}. In the process, we
formulate our idea leading to the notion of {\it frame multiplication}, 
that is used to define the vector-valued
ambiguity function.
In Section
\ref{sec:stftform} we also give two diverse examples.
The first is for {\it DFT} frames
(Subsection \ref{sec:apdftframe}), 
that we present in an Abelian setting.
The second is for
cross-product frames (Subsection \ref{sec:crossprod}), 
that is fundamentally non-Abelian and non-group with regard to structure,
and that is motivated by the recent applicability of quaternions, e.g., \cite{KriNicGra2009}. 
Subsection \ref{sec:formala1def} 
relates the examples of Subsections \ref{sec:apdftframe} and 
\ref{sec:crossprod}, and formally 
motivates the theory of
frame multiplication presented in Section \ref{sec:fm}.

In Section \ref{sec:harmgroupframes} we define the 
{\it harmonic} and {\it group} frames that are the basis 
for our Abelian group frame multiplication results of Section \ref{sec:fmgroupframes}. 
Although we present the
results in the setting of finite Abelian groups and $d$-dimensional Hilbert spaces, many
of them can be generalized; and, in fact, some are more easily formulated and
proved in the general setting. As such, some of the theory in these sections is given in infinite
and/or non-Abelian terms. The major results are stated and proved in Subsection \ref{sec:abelfm}.
They characterize the existence of frame multiplication in term of harmonic and group
frames.

Section \ref{sec:uncertainty} is devoted to the uncertainty principle in the context of our
vector-valued {\it DFT} theory.

We close with Appendix \ref{sec:unirep}. Some of this material is used explicitly in Sections
\ref{sec:harmgroupframes} and \ref{sec:fmgroupframes},
and some provides a theoretical umbrella to cover the theory herein and the transition to the 
non-Abelian case
beginning with \cite{andr2017}.

\begin{rem}
The forthcoming non-Abelian theory is due to Travis Andrews \cite{andr2017}.
In fact, if $(G,\bullet)$ is a finite group 
with representation
$\rho:G\rightarrow GL(\CC^d)$, then we can show that there is a frame $\{x_n\}_{n\in G}$
and bilinear multiplication,
$\ast:\CC^d\times \CC^d\rightarrow \CC^d$, such that $x_m\ast x_n=x_{m\bullet n}$.

Further, we are extending the theory 
to tight frames for 
finite dimensional Hilbert spaces $H$ over $\CC$ and finite rings $G$,
so that there are meaningful
generalizations of the vector-valued $A_p^d(u)$ theory in the
formal but motivated settings of Equations \eqref{eq:ap1def} and \eqref{eq:apddef}.

 It remains to establish the theory in infinite dimensional
 Hilbert spaces and associated infinite locally compact groups and rings
 as well as tantalizing
 non-group cases, see, e.g., our cross product
example in Subsection \ref{sec:crossprod} and its relationship to quaternion groups.

\end{rem}

\section[]{Frames}
\label{sec:frames}

\subsection{Definitions and properties}
\label{defframe}
Frames are a generalization of orthonormal bases where we relax Parseval's identity to allow for overcompleteness. 
Frames were first introduced in 1952 by Duffin and Schaeffer \cite{DufSch1952} and have become 
the subject of intense study since the 1980s. e.g., see \cite{daub1992}, \cite{BenWal1994}, \cite{bene1994}, 
\cite{chri2016}, \cite{CasKut2013}.
(In fact, Paley and Wiener gave the technical 
definition of a frame in \cite{PalWie1934}, but 
they only developed the completeness properties.)

\begin{defn}[Frame]
\label{def:frame}
{\it a.} Let $H$ be a separable Hilbert space over the
field $\FF,$ where $\FF = \RR$ or $\FF = \CC.$ A finite or countably infinite sequence, 
$X = \{x_j\}_{j \in J},$ of elements of $H$ 
is a {\it frame} for $H$ if 
\begin{equation}
\label{eq:framedef}
      \exists A, B > 0  \; \text{such that}  \;
      \forall x \in H, \quad A\norm{x}^2 \leq \sum_{j \in J} |\langle {x},{x_j}\rangle|^2 \leq B\norm{x}^2.
\end{equation}
The optimal constants, viz., the supremum over all such $A$ and infimum over all such $B$, are 
called the {\it lower} and {\it upper frame bounds} respectively. When we refer to {\it frame bounds} 
$A$ and $B$, we shall mean these optimal constants.

{\it b.} A frame $X$ for $H$ is a {\it tight frame} if $A = B.$ If a tight frame has the further property 
that $A = B = 1,$ then the frame is a {\it Parseval frame} for $H.$  

{\it c.}  A frame $X$ for $H$ is {\it equal-norm} 
if each of the elements of $X$ has the same norm. Further, a frame $X$ for $H$ is a {\it unit norm tight frame} 
(UNTF) if each of the elements 
of $X$ has norm $1.$  If $H$ is finite dimensional and $X$ is an UNTF for $H,$ then $X$ is a 
{\it finite unit norm tight frame} ({\it FUNTF}).

{\it d.} A sequence of elements of $H$ satisfying an upper frame bound,
such as $B\norm{x}^2$ in (\ref{eq:framedef}), is a {\it Bessel sequence}.
\end{defn}

\begin{rem}
The series in (\ref{eq:framedef}) is an absolutely convergent series of positive numbers; and so, any reordering of
 the sequence of frame elements or reindexing by another set of the same cardinality will remain a frame. We allow for 
 repetitions of vectors in a frame so that, strictly speaking, the set of vectors, that we also call $X$, is a multi-set. We shall 
 index frames by an arbitrary sequence such as $J$ in the definition, or by specific sequences such as the set 
 ${\mathbb N}$ of positive 
integers 
or the set ${\mathbb Z}^d, d \geq 2,$ of multi-integers when it is natural to do so.
\end{rem}

Let $X = \{x_j\}_{j \in J}$ be a frame for $H$. We define 
the following operators associated with every frame; they 
are crucial to frame theory and will be used extensively. The {\it analysis operator} $L : H \rightarrow \ell^2(J)$ is defined by
\[
       \forall x \in H,  \quad Lx = \{\inner{x}{x_j} \}_{j \in J}.
\]
Inequality $\eqref{eq:framedef}$ ensures that the analysis operator $L$ is bounded.
If $H_1$ and $H_2$ are separable Hilbert spaces and if $T : H_1 \rightarrow H_2$ is a 
linear operator, then the {\it operator norm} $\norm{T}_{op}$ of $T$ is
\[
   \norm{T}_{op} = {\rm sup}_{\norm{x}_{H_1} \leq 1} \norm{T(x)}_{H_2}.
\]
Clearly, we have $\norm{L}_{op} \leq \sqrt{B}$.
The adjoint of the analysis operator is the {\it synthesis operator} 
$L^\ast : \ell^2(J) \rightarrow H$, and it is defined by
\[
        \forall a \in \ell^2(J),  \quad L^\ast a= \sum_{j \in J} a_j x_j.
\]
From Hilbert space theory, we know that any bounded linear operator $T: H \rightarrow H$ 
satisfies $\norm{T}_{op} = \norm{T^\ast}_{op}.$  Therefore, the synthesis operator $L^\ast$ is bounded 
and $\norm{L^\ast}_{op} \leq \sqrt{B}$.

The {\it frame operator} is the mapping $S : H \rightarrow H$ defined as $S = L^\ast L$, i.e., 
\[
\forall x \in H, \quad Sx = \sum_{j \in J} \inner{x}{x_j}  x_j.
\]
We shall  describe $S$ more fully. First, we have that
\[
\forall x \in H, \quad \inner{Sx}{x}= \sum_{j \in J} \left| \inner{x}{x_j}\right|^2.
\]
Thus, $S$ is a positive and self-adjoint operator, and $\eqref{eq:framedef}$ can be rewritten as
\begin{equation*}
        \forall x \in H, \quad A\norm{x}^2 \leq \inner{Sx}{x} \leq B\norm{x}^2,
\end{equation*}
or, more compactly, as
\begin{equation*}
AI \leq S \leq BI.
\end{equation*}
It follows that $S$ is invertible (\cite{daub1992}, \cite{bene1994}), $S$ is a multiple of the identity precisely when $X$ is a tight frame, and 
\begin{equation}\label{eq:frameopinverse}
B^{-1}I \leq S^{-1} \leq A^{-1} I.
\end{equation}
Hence, $S^{-1}$ is a positive self-adjoint operator and has a square root $S^{-1/2}$ (Theorem 12.33 in \cite{rudi1991}). This square root can be written as a power series in $S^{-1}$; consequently, it commutes with every operator 
that commutes with $S^{-1},$ and, in particular, with $S.$ Utilizing these facts we can prove a theorem that tells us
that frames share an important property with orthonormal bases, viz., there is a reconstruction formula.

\begin{thm}[Frame reconstruction formula]
Let $H$ be a separable Hilbert space, and let $X = \{x_j\}_{j \in J}$ be a frame for $H$ with frame operator $S$. Then 
\[
        \forall x \in H, \quad x = \sum_{j \in J} \inner{x}{x_j}S^{-1} x_j = \sum_{j \in J} \inner{x}{S^{-1}x_j} x_j 
        = \sum_{j \in J} \inner{x}{S^{-1/2}x_j}S^{-1/2}x_j,
\]
where the mapping $S:H \rightarrow H,$ $x \mapsto \sum_{j \in J} \inner{x}{x_j}  x_j,$
is a well-defined topological isomorphism.
\begin{proof}
The proof is three computations. From $I = S^{-1}S$, we have
\[
\forall x \in H, \quad x = S^{-1}Sx = S^{-1} \sum_{j \in J} \inner{x}{x_j}x_j = \sum_{j \in J} \inner{x}{x_j}S^{-1} x_j;
\]
from $I = SS^{-1}$, we have
\[
\forall x \in H, \quad x = SS^{-1}x = \sum_{j \in J} \inner{S^{-1}x}{x_j}x_j = \sum_{j \in J} \inner{x}{S^{-1}x_j}x_j;
\]
and from $I = S^{-1/2}SS^{-1/2}$, it follows that
\[
       \forall x \in H, \, x = S^{-1/2}SS^{-1/2}x = S^{-1/2} \sum_{j \in J} \inner{S^{-1/2}x}{x_j}x_j 
       = \sum_{j \in J} \inner{x}{S^{-1/2}x_j}S^{-1/2}x_j.
       \qedhere
\]
\end{proof}
\end{thm}

From the frame reconstruction formula and \eqref{eq:frameopinverse}, it follows that 
$\{S^{-1}x_j\}_{j \in J}$ is a frame with frame bounds $B^{-1}$ and $A^{-1}$ and 
$\{S^{-1/2}x_j\}_{j \in J}$ is a Parseval frame.
\begin{defn}[Canonical dual]
Let $X = \{x_j\}_{j \in J}$ be a frame for a separable Hilbert space $H$ with frame operator $S$. 
The frame $S^{-1}X = \{S^{-1}x_j\}_{j \in J}$ is the {\it canonical dual frame} of $X$. The frame $S^{-1/2}X = 
\{S^{-1/2}x_j\}_{j \in J}$ is the {\it canonical tight frame} of $X$.
\end{defn}

The {\it Gramian operator} is the mapping  $G : \ell^2(J) \rightarrow \ell^2(J)$ defined as 
$G = L L^\ast$. If $\{x_j\}_{j \in J}$ is the standard orthonormal basis for $\ell^2(J)$, then
\begin{equation}
\label{eq:gramian}
    \forall a = \{a_j\}_{j \in J} \in \ell^2(J), \quad \inner{Ga}{x_k} = \sum_{j \in J}a_j \inner{x_j}{x_k}.
\end{equation}

\subsection{FUNTFs}
\label{sec:funtfs}

We shall often deal with {\it FUNTFs}
$X=\{x_j\}_{j=1}^N $
for $\CC^d.$ 

The most interesting setting is for the case when $N > d.$ In fact, 
frames can provide redundant signal representation to compensate
for hardware errors, can ensure numerical stability, and are a natural model for
minimizing the effects of noise. Particular areas of a recent applicability of
{\it FUNTFs} include the following topics:
\begin{itemize}

\item  Robust transmission of data over erasure channels such as
the internet, e.g., see \cite{GoyKovVet1999}, \cite{GoyKovKel2001}, \cite{CasKov2003};

\item  Multiple antenna code design for wireless communications,
e.g., see \cite{HocMarRicSwe2000};

\item  Multiple description coding, e.g., see \cite{GoyKovVet1998}, \cite{HeaStr2003};

\item Quantum detection, e.g., see \cite{forn1991}, \cite{BolEld2003}, \cite{BenKeb2008}; 
\item  Grassmannian ``min-max" waveforms, e.g., see 
\cite{CalHarRai1999}, \cite{HeaStr2003}, \cite{BenKol2006}.

\end{itemize}
The following is a consequence of (\ref{eq:framedef}).

\begin{thm}

 If $X = \{x_j\}_{j=0}^{N-1}$ is a {\it FUNTF} for $\FF^d$, then
$$
   \forall x \in \FF^d, \quad x = \frac{d}{N}\sum_{j=0}^{N-1}\langle x,x_j\rangle x_j.
$$
\end{thm}

\begin{rem}
It is important to understand the geometry of {\it FUNTFs}.
e.g., at the most elementary level, the vertices of the Platonic solids centered at the origin
are {\it FUNTFs}. Further, {\it FUNTFs} can be characterized as the minima of a 
potential energy function, see \cite{BenFic2003} for the details of this result.

Orthonormal bases for $H = \FF^d$ are both Parseval frames and {\it FUNTFs}. 
If $X = \{x_j\}_{j=0}^{N-1}$ is Parseval for 
$H$ and each $\|x_j\| = 1,$ then $N=d$ and $X$ is an ONB for $H$. If $X$ 
is a {\it FUNTF} with
frame constant $A,$ then $A \neq 1$ if $X$ is not an ONB. Further, a {\it FUNTF} 
 $X$ is not
a Parseval frame  unless $N=d$ and $X$ is an ONB; and, similarly, 
a Parseval frame is not a {\it FUNTF} unless $N=d$ and $X$ is an ONB.

Let $X = \{x_j\}_{j=0}^{N-1}$ be a Parseval frame. Then, each $\|x_j\| \leq 1.$ If $X$ is also equiangular, 
then each $\|x_j\| < 1,$
whereas we can not conclude that any $\|x_j\|$ ever equals an $\|x_k\|$ unless $j=k$.
\end{rem}

When $H$ is finite dimensional, e.g., $H = \FF^d$, and $X = \{x_j\}_{j=0}^{N-1}$, 
then each of the above operators 
can be 
realized as multiplication on the left by a matrix. The synthesis operator, 
$L^\ast$, is the $d \times N$ matrix with the 
frame elements as its columns, i.e., 
\begin{equation*}
      L^\ast = \left[\begin{array}{c|c|c|c}
      x_0 & x_1 & \ldots & x_{N-1}
     \end{array}\right];
\end{equation*}
and the analysis operator, $L$,  is the $N \times d$ matrix with the conjugate transposes $x_j^\ast$
of the frame elements $x_j$ as its rows, i.e.,
\begin{equation*}
L = \left[\begin{array}{c}
x_0^\ast \\
x_1^\ast \\
\vdots \\
x_{N-1}^\ast
\end{array}\right].
\end{equation*}
The frame operator and Gramian are the products of these matrices. From direct multiplication of $LL^\ast$ or \eqref{eq:gramian} it is apparent that the Gramian, or Gram matrix, has entries
\begin{equation*}
G_{jk} = \inner{x_k}{x_j}.
\end{equation*}


\subsection{Naimark's theorem}
\label{naimark}
The following theorem, a weak variant of Naimark's dilation theorem, tells us every Parseval frame is the 
projection of an orthonormal basis in a larger space. The general form of Naimark's dilation theorem is a result 
for an uncountable family of increasing operators on a Hilbert space satisfying some additional conditions. 
It states that it is possible to construct an embedding into a larger space such that the dilations of the operators to 
this larger space commute and are a resolution of the identity. For an excellent description of this dilation problem 
and an independent geometric proof of a finite version of Naimark's dilation theorem we recommend an article by 
C. H. Davis, \cite{davi1977}. To see the connection of this general theorem with the one below, consider the 
finite sums of the (rank one)  projections onto the subspaces spanned by elements of a Parseval frame. 

\begin{thm}[Naimark's theorem, e.g., \cite{AkhGla1966}, \cite{HanLar2000}]
\label{thm:naimark}
A set $X = \{x_j\}_{j \in J}$ in a Hilbert space $H$ is a Parseval frame for $H$ if and only if there is a 
Hilbert space $K$ containing $H$ and an orthonormal basis $\{e_j\}_{j \in J}$ for $K$ such that the 
orthogonal projection $P$ of $K$ onto $H$ satisfies
\[
     \forall j \in J, \quad Pe_j = x_j.
\]

\begin{proof}
Let $X = \{x_j\}_{j \in J}$ be a Parseval frame for $H$, let $K = \ell^2(J)$, and let $L$ be the analysis 
operator of $X$. Since $X$ is a Parseval frame for $H$, we have
\[
     \norm{Lx}^{2}_K = \sum_{j \in J} \left| \inner{x}{x_j}\right|^2 = \norm{x}^{2}_H.
\]
Thus, $L$ is an isometry, and we can embed $H$ into $K$ by identifying $H$ with $L(H)$. Let $P$ be the 
orthogonal projection from $K$ onto $L(H)$. Denote the standard orthonormal basis for $K$ by 
$\{e_j\}_{j \in J}$. We claim that $Pe_n = Lx_n$ for each $n \in J$. 
To this end, we take any $m \in J$, and make the following computation:
\begin{align}\label{eq:Naimarkproof}
\inner{Lx_m}{Pe_n}_K &= \inner{PLx_m}{e_n}_K = \inner{Lx_m}{e_n}_K  \notag\\
&= \inner{x_m}{x_n}_H = \inner{Lx_m}{L x_n}_K.
\end{align}
In \eqref{eq:Naimarkproof} we use the fact that $P$ is an orthogonal projection for the first equality, 
that $Lx_m$ is in the range of $P$ for the second, the definitions of $L$ and $\{e_j\}_{j \in J}$ for the 
third, and that $L$ is an isometry for the last. Rearranging \eqref{eq:Naimarkproof} yields 
\[
     \inner{Lx_m}{Pe_n - Lx_n}_K = 0.
\] 
Since the vectors $Lx_m$ span $L(H)$ it follows that $Pe_n - Lx_n \perp L(H)$, whereas 
$Pe_n - Lx_n \in L(H)$. 
Thus, $Pe_n - Lx_n = 0$ as claimed. 

For the converse, assume that $H \subseteq K$, $\{e_j\}_{j \in J}$ is an orthonormal basis for $K$, 
$P$ is the orthogonal projection of $K$ onto $H$, and $Pe_j = x_j$. We claim that $X = \{x_j\}_{j \in J}$ 
is a Parseval frame for $H$. For any $y \in K$, we have Parseval's identity 
\[
      \norm{y}^{2}_K = \sum_{j \in J}\left| \inner{y}{e_j}_K \right|^2.
\] 
For $x \in H$, we additionally have $Px = x$. Thus,
\[
      \forall x \in H, \quad \norm{x}^{2}_H = \sum_{j \in J}\left| \inner{x}{e_j}_K \right|^2 = 
      \sum_{j \in J}\left| \inner{Px}{e_j}_K \right|^2 = \sum_{j \in J}\left| \inner{x}{Pe_j}_H \right|^2,
\] 
i.e., $\{x_j\}_{j \in J} = \{Pe_j\}_{j \in J}$ is a Parseval frame for $H$.
\end{proof}
\end{thm}

\begin{rem}
If $X$ is a Parseval frame, then $L^\ast L = S = I$, and so $G^2 = LL^\ast L L^\ast = LL^\ast = G$. 
Hence, $G$ is a projection, and since it is self-adjoint it is an orthogonal projection. Furthermore, 
$Gx_j = LL^\ast x_j = Lx_j$. Thus, the orthogonal projection $P$ onto $L(H)$ from Naimark's theorem 
is precisely $G$.
\end{rem}


\subsection{DFT frames}
\label{sec:dftframes}

 The {\it characters} of the Abelian group $\ZZ/N\ZZ$ 
are the functions $\{\gamma_n\}, n = 0, \ldots, N-1,$ defined by $m \mapsto e^{2 \pi i mn / N}$, so that the 
dual $(\ZZ/N\ZZ\widehat{)}$ is isomorphic to $\ZZ/N\ZZ$ under the identification $\gamma_n \mapsto n.$
Hence, the Fourier transform on $\ell^2(\ZZ/N\ZZ) \simeq \CC^N$ is a linear map that can be expressed as
\begin{equation}
\label{eq:dftdef}
        \forall n \in \ZZ/N\ZZ, \quad \widehat{x}(n) = \sum_{m = 0}^{N-1} x(m)e^{-2 \pi i m n /N}.
\end{equation}

It is elementary to see that the Fourier transform is defined by a linear
transformation whose matrix representation is
\begin{equation}
\label{eq:dftmatrix}
       D_N = (e^{-2\pi i m n /N})_{m,n = 0}^{N-1}.
\end{equation}
The Fourier transform on $\CC^N$ is called the {\it discrete Fourier transform} ({\it DFT}), and $D_N$ 
is the {\it {\it DFT} matrix}. The {\it DFT} has applications in digital signal processing 
and a plethora of numerical algorithms. Part of the reason why its use is so ubiquitous is that fast algorithms 
exist for its computation. The {\it Fast Fourier Transform} ({\it FFT}) allows the computation of the {\it DFT} to take 
place in $O(N\log N)$ operations. This is a significant improvement
 over the $O(N^2)$ operations it would take to 
compute the {\it DFT} directly by means of \eqref{eq:dftdef}. The fundamental paper on the {\it FFT }
is due to Cooley and Tukey \cite{CooTuk1965}, 
in which they describe what is now referred to as the Cooley-Tukey {\it FFT} algorithm. The algorithm employs a 
divide and conquer method going back to Gauss
to break the $N$ dimensional {\it DFT} into smaller {\it DFT}s that may then be further 
broken down, computed, and reassembled. For a more extensive description of the {\it DFT}, {\it FFT}, and their 
relationship to sampling, sparsity,  and the Fourier transform on $\ell^1(\ZZ),$ see, e.g., \cite{bene1997},
\cite{terr1999}, and
\cite{GilIndIweSch2014}.

\begin{defn}[{\it DFT} frame]
\label{defn:dftframe}
Let $N\geq d$, and let $s: \ZZ/d\ZZ \rightarrow \ZZ/N\ZZ$ be injective.
For each $m = 0, \ldots, N-1,$ set
\[
      x_m = \left(e^{2\pi ims(1)/N}, \ldots, e^{2\pi ims(d)/N} \right) \in \CC^d,
\]
and define the $N \times d$ matrix,
\[
      \left(e^{2 \pi ims(n)/N}\right)_{m,n}.
\]
Then $X = \{x_m\}_{m=0}^{N-1}$ denotes its $N$ rows, and it is
an equal-norm tight frame for $\CC^d$ called a {\it {\it DFT} frame}. 

The name comes from the 
fact that the elements of $X$ are projections of the rows of the conjugate of the ordinary {\it DFT} 
matrix \eqref{eq:dftmatrix}. That $X$ is an equal-norm tight frame follows from Naimark's theorem 
(Theorem \ref{thm:naimark}) and the fact that the {\it DFT} matrix has orthogonal columns. In fact,
$(1/\sqrt N)\,D_N$ is a unitary matrix.
\end{defn}

The rows of the $N \times d$ matrix in Definition \ref{defn:dftframe}, up to multiplication by ${1}/{\sqrt {d}},$
form a {\it FUNTF} for $\CC^d.$ For example, see Figure \ref{fig:dftfuntf}, where $d=5$ and $N=8$. The function,
$s$ of Definition \ref{defn:dftframe} determines the 5 columns
of $\ast$s, that, in turn, determine ${\mathbb C}^5$.
\begin{figure}[htbp]
\centering
\centerline{\resizebox{8cm}{6cm}
{\includegraphics{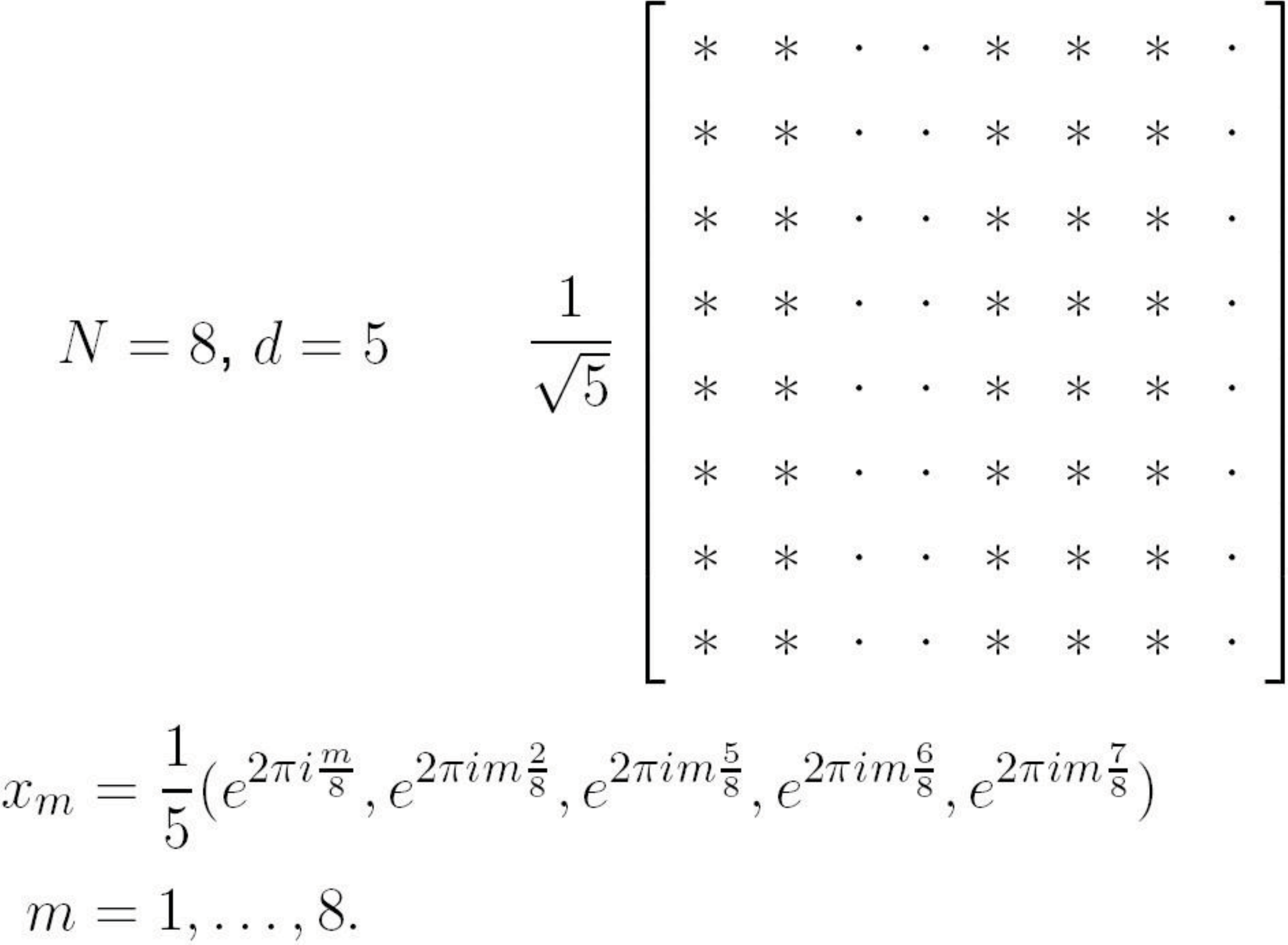}}}
\caption{A {\it {\it DFT} FUNTF}}
\label{fig:dftfuntf}
\end{figure}

Also, for a given $N,$
we shall use the notation, $\omega = e_{-1} = e^{-2\pi i/N},$ and so $e_m = e^{2\pi im/N}$.
Note that $\{e_m\}_{m=0}^{N-1}$ is a tight frame for $\CC.$


\section{The vector-valued discrete Fourier transform (DFT)}
\label{sec:vvDFT}

\subsection{Definition and inversion theorem}
\label{sec:dftdefinv}

In order to achieve the goals listed in Subsection \ref{sec:goals}, we shall also have to develop 
a vector-valued {\it DFT} theory to 
{\it verify}, not just {\it motivate},
that $A_{p}^{d}(u)$ is an STFT in the case $\{x_k\}_{k=0}^{N-1}$ is a {\it DFT}
frame for $\CC^d$.

We shall use the convention that the juxtaposition of vectors of equal dimension is the pointwise 
product of those vectors. Thus, for two functions, $u, v : \ZZ/N\ZZ \rightarrow \mathbb C^d,$ we let 
$uv$ be the coordinate-wise product of $u$ and $v.$ This means that
\[
     \forall m \in \ZZ/N\ZZ, \quad (uv)(m) = u(m)v(m) \in \mathbb C^d,
\]
where the product on the right is pointwise multiplication of vectors in $\ell^2(\ZZ/d\ZZ),$
and so $ u(m)(p), v(m)(p), (uv)(m)(p) \in \CC$ for each $p \in \ZZ/d\ZZ,$
i.e., $u(m)(p)$ designates the $p$th coordinate in $\mathbb C^d$ of the vector 
$u(m) \in \mathbb C^d.$

\begin{defn}[Vector-valued discrete Fourier transform]
Let $\{x_k\}_{k=0}^{N-1}$ be a {\it DFT} frame for $\mathbb C^d$
with injective mapping $s.$ Given $u:\mathbb \ZZ/N\ZZ \rightarrow \mathbb C^d,$
the {\em vector-valued discrete Fourier transform} (vector-valued {\it DFT}) $\widehat{u}$ of $u$ 
is defined by the formula,
\begin{equation}
\label{eq:vvDFTdef}
        \forall n \in \ZZ/N\ZZ, \quad F(u)(n) = \widehat{u}(n) =
        \sum_{m=0}^{N-1}u(m)  x_{-mn} \in \CC^d,
\end{equation}
where the product $u(m)x_{-mn}$ is pointwise (coordinate-wise) multiplication. Further, the mapping 
\[
      F: \ell^2(\ZZ/N\ZZ \times \ZZ/d\ZZ) \rightarrow \ell^2(\ZZ/N\ZZ \times \ZZ/d\ZZ)
\]
is a linear operator.
\end{defn}

\begin{rem}
{\it a.} Given $u : \ZZ/N\ZZ \rightarrow \CC^d.$ 
We write $u \in \ell^2(\ZZ/N\ZZ \times \ZZ/d\ZZ)$ as a function of two arguments so that 
$u(m)(p) \in \CC.$  With this notation we can think of $u$ and $\widehat{u}$ as $N\times d$ 
matrices with entries $u(m)(p)$ and $\widehat{u}(n)(q),$ respectively.

{\it b.} Thus, we have
\begin{align*}
        \forall q \in \ZZ/d\ZZ, \quad \widehat{u}(n)(q) &= \left( \sum_{m=0}^{N-1}u(m)x_{-mn}\right)(q)\\
        &= \left( \sum_{m=0}^{N-1}u(m)(q)x_{-mn}(q)\right).
\end{align*}
From this we see that $\widehat{u}(n)(q)$ depends only on $\{u(m)(q)\}_{m=0}^{N-1}$, i.e., 
when thought of as matrices the $q$-th column of $\widehat{u}$ depends only on the $q$-th column of $u$.
\end{rem}

\begin{thm}[Inversion theorem]
\label{thm:vvDFTinversion}
The vector-valued {\it DFT} is invertible if and only if $s$, the injective function defining the 
{\it DFT} frame, has the property that
\[\forall n \in \ZZ/d\ZZ, \quad (s(n),N) = 1.\]
In this case, the inverse is given by
\[
  \forall \ m \in \mathbb \ZZ/N\ZZ, \quad u(m) = (F^{-1}\widehat{u})(m) = \frac{1}{N}\sum_{p=0}^{N-1}\widehat{u}(p)x_{mp};
\]
and we also have that $F^\ast F = F F^\ast = NI$, where $I$ is the identity operator.

\begin{proof}
We first show the forward direction. Suppose there is $n_0 \in \ZZ/d\ZZ$ such that $(s(n_0), N) \neq 1$. 
Then there exists $j, l, M \in \NN$ such that $j >1$, $s(n_0) = jl$, and $N = jM$. Define a matrix $A$ as
\[
     A = (e^{2 \pi i m k s(n_0)/N})_{m,k = 0}^{N-1} = (e^{2 \pi i m k l/M})_{m,k = 0}^{N-1}.
\]
$A$ has rank strictly less than $N$ since the $0$-th and $M$-th rows are all $1$s. Therefore we can 
choose a vector $v \in \CC^N$ orthogonal to the rows of $A$. Define $u: \ZZ/N\ZZ \rightarrow \CC^d$ by
\[
u(m)(n) = \begin{cases}
v(m) &\text{ if } n = n_0 \\
0 &\text{ otherwise.}
\end{cases}
\]
Then,
\begin{align*}
     \forall n \neq n_0, \quad \widehat{u}(m)(n) 
      &= \sum_{k=0}^{N-1} u(k)(n) x_{-mk}(n) = \sum_{k=0}^{N-1} 0 \cdot x_{-mk}(n) = 0,
\end{align*}
while, for $n = n_0$, we have
\[
\widehat{u}(m)(n_0) = \sum_{k=0}^{N-1} u(k)(n_0) x_{-mk}(n_0) 
= \sum_{k=0}^{N-1} u(k)(n_0) e^{- 2 \pi i m k s(n_0)/N}
\]
\[
= \sum_{k=0}^{N-1} u(k)(n_0) e^{-2 \pi i m k l/M}
= \inner{u(\cdot)(n_0)}{e^{2 \pi i m (\cdot) l /M}}
= \inner{v}{e^{2 \pi i m (\cdot) l /M}}
= 0.
\]
The final equality follows from the fact that $v$ is orthogonal to the rows of $A$. Hence, the vector-valued {\it DFT} 
defined by $s$ has non-trivial kernel and is not invertible.

We prove the converse and the formula for the inverse with a direct calculation. We compute
\begin{align*}
\sum_{n=0}^{N-1} \widehat{u}(n)x_{mn} &= \sum_{n=0}^{N-1}\left(\sum_{k=0}^{N-1} u(k)x_{-kn} \right)  x_{mn}\\
&= \sum_{k=0}^{N-1} \left(u(k)\left( \sum_{n=0}^{N-1} x_{n(m-k)}\right) \right).
\end{align*}
The $r$-th component of the last summation is
\begin{align*}
\sum_{n=0}^{N-1} x_{n(m-k)}(r) &= \sum_{n=0}^{N-1} e^{2 \pi i n(m-k)s(r)/N}\\
&= \begin{cases}
N & \text{if } (m-k)s(r) \equiv 0 \text{ mod }N\\
0 & \text{if } (m-k)s(r) \not\equiv 0 \text{ mod } N.
\end{cases}
\end{align*}
Since $(s(r),N) = 1$, the first cases occurs if and only if $k=m$. Continuing with the previous calculation, we have
\[\sum_{k=0}^{N-1} \left(u(k)\left( \sum_{n=0}^{N-1} x_{n(m-k)}\right) \right) = N u(m).\]

Finally, we compute the adjoint of $F$.
\[
\inner{Fu}{v} = \sum_{m=0}^{N-1} \sum_{n=0}^{d-1} \widehat{u}(m)(n) \overline{v(m)(n)}
= \sum_{m=0}^{N-1} \sum_{n=0}^{d-1} \left(\sum_{k=0}^{N-1} u(k)(n)x_{-mk}(n) \right)\overline{v(m)(n)}
\]
\[
= \sum_{m=0}^{N-1} \sum_{n=0}^{d-1} \left(\sum_{k=0}^{N-1} u(k)(n)e^{-2\pi i mk s(n) / N} \right)\overline{v(m)(n)}
= \sum_{k=0}^{N-1} \sum_{n=0}^{d-1} \left(\sum_{m=0}^{N-1} \overline{v(m)(n)e^{2\pi i mk s(n) / N}} \right)u(k)(n)
\]
\[
= \sum_{k=0}^{N-1} \sum_{n=0}^{d-1} u(k)(n) \overline{\left(\sum_{m=0}^{N-1} v(m)(n)x_{mk}(n) \right)}
= \inner{u}{F^*v}.
\]
Therefore, $F^*$ is defined by
\[(F^* v )(k) = \sum_{m=0}^{N-1} v(m)x_{mk},\]
and $F^* = N F^{-1}$.
\end{proof}
\end{thm}

By Theorem \ref{thm:vvDFTinversion}, we can define the {\it unitary vector-valued discrete Fourier transform} 
$\Fcal$ by the formula
\[
         \Fcal = \frac{1}{\sqrt{N}} F.
\]
With this definition, we have
\[
        \Fcal \Fcal^\ast = \Fcal^\ast \Fcal = I,
\]
and $\Fcal$ is unitary
.
\begin{defn}[Translation and modulation]
\label{defn:transmod}
Let $u: \ZZ/N\ZZ \rightarrow \mathbb C^d$, and let $\{x_k\}_{k=0}^{N-1}$ be a {\it DFT} frame for 
$\mathbb C^d$. For each $j \in \ZZ/N\ZZ$, define the {\it translation operators}, 
\[
      \tau_j : \ell^2(\ZZ/N\ZZ \times \ZZ/d\ZZ) \rightarrow \ell^2(\ZZ/N\ZZ \times \ZZ/d\ZZ), \quad
      \tau_{j}u(m) = u(m-j),
\]
and the {\it modulation operators},
\[
       e^j :  \ZZ/N\ZZ \rightarrow \mathbb C^d, \quad  e^{j}(k) = x_{jk}.
\]
\end{defn}

The usual translation and modulation properties of the Fourier transform hold
for the vector-valued transform.

\begin{thm}
\label{thm:vvdfttm}
Let $u: \ZZ/N\ZZ \rightarrow \mathbb C^d$, and let
$\{x_k\}_{k=0}^{N-1}$ be a {\it DFT} frame for $\mathbb C^d$ with associated vector-valued 
discrete Fourier transform $F.$
Then,
$$
     F(\tau_{j}u)= e^{-j}\widehat{u}
$$
and
$$
    F(e^{j}u) = \tau_{j}\widehat{u}.
$$
\end{thm}

\begin{proof}
{\it i.} We compute
\[
\widehat{\tau_j u}(n) = \sum_{m=0}^{N-1} \tau_j u(m)x_{-mn} 
= \sum_{m=0}^{N-1} u(m-j)  x_{-mn}
= \sum_{k=-j}^{N-1-j} u(k)  x_{-(k+j)n}
\]
\[
= \sum_{k=0}^{N-1} u(k)  x_{-kn - jn}
= x_{-jn}\left(\sum_{k=0}^{N-1} u(k)  x_{-kn}\right)
= x_{-jn} \widehat{u}(n).
\]
The third equality follows by setting $k = m - j$, the fourth by reordering the sum and noting 
that the index of summation is modulo $N$, and the fifth follows since $x_{j+k} = x_j  x_k$ and 
by the bilinearity of pointwise products.

{\it ii.} We compute
\[
\widehat{e^j  u}(n) = \sum_{m=0}^{N-1} (e^j  u)(m)  x_{-mn}
= \sum_{m=0}^{N-1} x_{jm} u(m)x_{-mn}
\]
\[
= \sum_{m=0}^{N-1} u(m)x_{-m(n-j)}\\
= \widehat{u}(n-j).
\]
The third equality follows from commutativity and since $x_{j+k} = x_j  x_k$.
\end{proof}


\subsection{A matrix formulation of the vector-valued DFT}
\label{sec:altvvdft}

We now describe a different way of viewing the vector-valued {\it DFT} that makes some properties more apparent. 
Given $N \in \NN$, define the matrices $\Dcal_\ell,$ 
\[
      \forall \ell \in \ZZ/N\ZZ, \quad \Dcal_\ell = (e^{-2 \pi i m n \ell / N})_{m,n=0}^{N-1}. 
\]
By definition of the vector-valued {\it DFT}, we have
\[
         \widehat{u}(n)(q) = \left( \sum_{m=0}^{N-1}u(m)(q)x_{-mn}(q)\right) 
\]
\[
           = \left( \sum_{m=0}^{N-1}u(m)(q)e^{-2 \pi i mn s(q)/N}\right) 
             = \left(\Dcal_{s(q)}u(\cdot)(q)\right) (n),
\]
i.e., the vector $\widehat{u}(\cdot)(q)$ is equal to the vector $\Dcal_{s(q)} u(\cdot)(q)$. In other words, 
we obtain $\widehat{u}$ by applying the matrix $\Dcal_{s(q)}$ to the $q$--th column of $u$ for each 
$0 \leq q \leq d-1$. Therefore, $F$ is invertible if and only if each matrix $\Dcal_{s(q)}$ is invertible. 

The rows of $\Dcal_{\ell}$ are a subset of the rows of the {\it DFT} matrix, and each row of the {\it DFT} matrix 
is a character of $\ZZ/N\ZZ$. Taken as a collection, the characters form the dual group 
$(\ZZ/N\ZZ\widehat{)} \simeq \ZZ/N\ZZ$ under pointwise multiplication. With this group operation and the fact that
\[\forall m, n \in \ZZ/N\ZZ, \quad e^{-2\pi i m n \ell /N} = (e^{-2\pi i n \ell/N})^m,\]
we see the rows of $\Dcal_{\ell}$ are the orbit of some element $\gamma \in (\ZZ/N\ZZ\widehat{)}$ 
repeated $|\gamma |/N$ times. Hence, $\Dcal_\ell$ is invertible if and only if $\gamma$ generates the 
entire dual group. From the theory of cyclic groups, $\gamma$ is a generator of $(\ZZ/N\ZZ\widehat{)}$ 
if and only if $\gamma = (e^{-2\pi i n \ell/N})_{n = 0}^{N-1}$ for some $\ell$ relatively prime to $N$. Therefore, 
$F$ is invertible if and only if $s(q)$ is relatively prime to $N$ for each $q$.

\begin{example}
Let $N = 4$ and recall that $\omega = e^{-2 \pi i /4}$. We compute the matrices $\Dcal_1,$
$\Dcal_2,$ and $\Dcal_3.$

\[\Dcal_1 = \left(
\begin{array}{cccc}
1 & 1 & 1 & 1\\
1 & \omega & \omega^2 & \omega^3\\
1 & \omega^2 & 1 & \omega^2\\
1 & \omega^3 & \omega^2 & \omega\\
\end{array}\right)
\quad
\Dcal_2 = \left(
\begin{array}{cccc}
1 & 1 & 1 & 1\\
1 & \omega^2 & 1 & \omega^2\\
1 & 1 & 1 & 1\\
1 & \omega^2 & 1 & \omega^2\\
\end{array}\right)
\quad
\Dcal_3 = \left(
\begin{array}{cccc}
1 & 1 & 1 & 1\\
1 & \omega^3 & \omega^2 & \omega\\
1 & \omega^2 & 1 & \omega^2\\
1 & \omega^1 & \omega^2 & \omega^3\\
\end{array}\right)\]
It is easy to see that $\Dcal_1$ and $\Dcal_3$ are invertible while $\Dcal_2$ is not invertible. 
In each case the matrix $\Dcal_i$ is generated by pointwise powers of its second row, which have orders 
$4$, $2$, and $4$ respectively. In fact, the full vector-valued {\it DFT} can be viewed as a block matrix, where the
$q$th block is $D_{s(q)}$.
\end{example}

\subsection{The Banach algebra of the vector-valued DFT}
\label{sec:vvBanachAlgebra}

We now study the vector-valued {\it DFT} in terms of Banach algebras.
In fact, we shall define a Banach algebra structure on $\Acal = L^1(\ZZ/N\ZZ \times \ZZ/d\ZZ)$, 
describe the spectrum $\sigma(\Acal)$ of $\Acal,$ and then prove that the Gelfand transform of $\Acal$ 
is the vector-valued {\it DFT}. 


To this end, first recall that if $G$ is a locally compact Abelian group (LCAG), then $L^{1}(G)$ is a 
commutative Banach algebra under convolution.

Next, let $\Bcal$ be a commutative Banach $\ast$-algebra over $\CC,$ where $\ast$ indicates the
{\it involution} satisfying the properties, $(x+y)^{\ast} = x^{\ast} + y^{\ast}, (cx)^{\ast} = \overline{ c}x^{\ast},
(xy)^{\ast} = y^{\ast}x^{\ast},$ and $x^{\ast\ast} = x$ for all
$x,y \in \Bcal$ and $c \in \CC.$ For example, let $ \Bcal = L^{1}(G)$ and define $f^{\ast}(t) = \overline{f(-t)}$
for $f \in L^{1}(G).$ The {\it spectrum} $\sigma(\Bcal)$ of $\Bcal$ is the set of non-zero homomorphisms,
$h : \Bcal \rightarrow \CC.$
$\sigma(\Bcal)$ is  subset of the weak $\ast$-compact unit ball of the dual space ${\Bcal}'$ of
the Banach space $\Bcal,$ and each $x \in \Bcal$ defines a function $\widehat{x} :  \sigma(\Bcal) \rightarrow \CC$
given by 
\[
      \forall h \in \sigma(\Bcal), \quad \widehat x(h) = h(x).
\]
$\widehat{x}$ is the {\it Gelfand transform} of $x.$ We shall use well-known properties of the Gelfand
transform, e.g., see \cite{pont1966}, \cite{GelRaiShi1964}, \cite{rudi1962},  \cite{HewRos1963}, 
\cite{HewRos1970}, \cite{reit1968}, \cite{foll1995}.

Using the group structure on $\ZZ/N\ZZ \times \ZZ/d\ZZ,$ 
we define the convolution of $u, v \in \Acal = L^1(\ZZ/N\ZZ \times \ZZ/d\ZZ )$ by the formula,
\begin{equation}
\label{eq:convolution1}
        (u \ast v)(m)(n) = \sum_{k = 0}^{N-1} \sum_{l = 0}^{d-1} u(k)(l)v(m-k)(n-l).
\end{equation}
This definition is not ideal for our purposes because it treats $u$ and $v$ as functions that take $Nd$ values. 
Our desire is to view $u$ and $v$ as functions that take $N$ values, that are each $d$ dimensional vectors. 
The convolution \eqref{eq:convolution1} can be rewritten as 
\begin{equation*}
         (u \ast v)(m)(n) = \sum_{k = 0}^{N-1} (u(k) \ast v(m-k))(n),
\end{equation*}
where the $\ast$ on the right hand side is $d$-dimensional convolution. Replacing this $d$-dimensional 
convolution with pointwise multiplication, we arrive at the following new definition of convolution 
on $\Acal.$
\begin{defn}[Vector-valued convolution]
Let $u, v \in \Acal.$ 
Define the {\it vector-valued convolution} of $u$ and $v$ by the formula
\[
            (u \ast_{\rm v} v)(m) = \sum_{k = 0}^{N-1} u(k)  v(m-k).
\]
\end{defn}

\begin{thm}
\label{thm:cba}
$\Acal$ equipped with the vector-valued convolution $\ast_{\rm v}$ is a 
commutative Banach $\ast$-algebra
with unit $e$ defined as
\[
      e(m) = \begin{cases}
     \vec{1} & m = 0\\
      \vec{0} & m \neq 0,
       \end{cases}
\]
where $\vec{1}$ and $\vec{0}$ are the vectors of $1$s and $0$s, respectively, and with involution 
defined as $u^\ast (m) = \overline{u(-m)}.$
\begin{proof}
It is essentially only necessary to verify that
$\norm{u \ast_{\rm v} v}_1 \leq \norm{u}_1 \norm{v}_1$ is valid.  We compute
\[
\norm{u\ast_{\rm v}  v}_1 = \sum_{m=0}^{N-1} \norm{u\ast v (m)}_{L^1(\ZZ/d\ZZ)} 
= \sum_{m=0}^{N-1} \norm{\sum_{k=0}^{N-1} u(k)  v(m-k) }_{L^1(\ZZ/d\ZZ)}
\]
\[
\leq \sum_{m=0}^{N-1} \sum_{k=0}^{N-1}  \norm{u(k)  v(m-k) }_{L^1(\ZZ/d\ZZ)}
\leq \sum_{m=0}^{N-1} \sum_{k=0}^{N-1}  \norm{u(k)}_{L^1(\ZZ/d\ZZ)} \norm{v(m-k) }_{L^1(\ZZ/d\ZZ)}
\]
\[
= \sum_{k=0}^{N-1} \norm{u(k)}_{L^1(\ZZ/d\ZZ)} \sum_{m=0}^{N-1} \norm{v(m-k) }_{L^1(\ZZ/d\ZZ)}\\
= \sum_{k=0}^{N-1} \norm{u(k)}_{L^1(\ZZ/d\ZZ)} \norm{v}_1\\
= \norm{u}_1 \norm{v}_1 .
\]
\qedhere
\end{proof}
\end{thm}

Tying this together with our {\it DFT} theory, we have the following
desired theorem relating $\mathcal{A}$ to the vector-valued {\it DFT}.

\begin{thm}[Convolution theorem]
Let $u, v \in \Acal$. The vector-valued Fourier transform of the convolution of $u$ 
and $v$ is the vector product of their Fourier transforms, i.e.,
\[F(u \ast_{\rm v} v) = F(u)  F(v).\]
\begin{proof}
\[
F(u \ast_{\rm v} v)(n) = \sum_{m=0}^{N-1} (u \ast v )(m)  x_{-mn}
= \sum_{m=0}^{N-1} \left( \sum_{k = 0}^{N-1} u(k)  v(m-k)  \right)  x_{-mn}
\]
\[
= \sum_{k=0}^{N-1} u(k)  \left(\sum_{m = 0}^{N-1} v(m-k)  x_{-mn}\right)
= \sum_{k=0}^{N-1} u(k)  \left(\sum_{l = 0}^{N-1} v(l)  x_{-(k+l)n}\right) 
\]
\[
= \left(\sum_{k=0}^{N-1} u(k)  x_{-kn}\right)  \left(\sum_{l = 0}^{N-1} v(l)  x_{-ln}\right) \\
= F(u)(n)  F(v)(n).
\]
\qedhere
\end{proof}
\end{thm}

We shall now describe the spectrum of $\mathcal{A}$ and the Gelfand transform of $\Acal$,
see Theorem \ref{thm:specgelfA}.

Define functions $\delta_{(i,j)}$ in $\Acal$ by
\[\delta_{(i,j)}(m)(n) = \begin{cases}
1 & (m,n) = (i,j) \\
0 & \text{otherwise}.\\
\end{cases}
\]
It is easy to see that $\delta_{(1,j)}^k = \delta_{(1,j)} \ast \ldots \ast \delta_{(1,j)} \text{ (k factors) } 
= \delta_{(k,j)}$ so that $\{\delta_{(1,j)}\}_{j = 0}^{d-1}$ generate $\mathcal{A}$. We shall find the 
spectrum of the individual elements of our generating set $\{\delta_{(1,j)}\}_{j = 0}^{d-1}$, and with 
this information describe the spectrum of $\Acal$.

To find the spectrum of $\delta_{(1,j)}$ we first find necessary conditions on $\lambda$ for 
$(\lambda e - \delta_{(1,j)})^{-1}$ to exist, and when these conditions are met we compute 
$(\lambda e - \delta_{(1,j)})^{-1}$ 
and thereby show the conditions are sufficient as well. To that end, suppose $u = (\lambda e - \delta_{(1,j)})^{-1}$ 
exists, i.e., $(\lambda e - \delta_{(1,j)}) \ast u = e$. Expanding the definitions on the left hand side
\begin{align*}
(\lambda e - \delta_{(1,j)}) \ast u (m) &= \sum_{k=0}^{N-1} (\lambda e - \delta_{(1,j)})(k)  u(m-k)\\
&= \lambda u(m) - \delta_{(1,j)}(1)  u(m-1).
\end{align*}
Setting the result equal to $e(m)$ and dividing into the cases $m = 0$ and $m \neq 0$ yields two equations
\begin{equation}\label{eq:inverse1}
\forall n \in \ZZ/d\ZZ, \quad \lambda u(0)(n) - \delta_{(1,j)}(1)(n) u(N-1)(n) = 1
\end{equation}
and
\begin{equation}\label{eq:inverse2}
\forall n \in \ZZ/d\ZZ \; {\rm and} \; \forall m \in {\ZZ/N\ZZ}\setminus\{0\}, \quad 
\lambda u(m)(n) - \delta_{(1,j)}(1)(n) u(m-1)(n) = 0.
\end{equation}
Substituting $n = j$ into \eqref{eq:inverse1} yields
\begin{equation}
\label{eq:inverse3}
       \lambda u(0)(j) - u(N-1)(j) = 1,
\end{equation}
while for $n \neq j$ we have
\[
       u(0)(n) = \frac{1}{\lambda}.
\]
Therefore, we must have $\lambda \neq 0$. Similarly, substituting $n = j$ in \eqref{eq:inverse2} gives
\begin{equation}
\label{eq:inverse4}
         \forall m \neq 0, \quad \lambda u(m)(j) - u(m-1)(j) = 0,
\end{equation}
while
\[
        \forall n \neq j, \; \forall m \neq 0, \quad u(m)(n) = 0.
\]
At this point and for our fixed $j$
we have specified all the values of $u$ except for $u(m)(j)$. 
Now, iterate \eqref{eq:inverse4} $N-1$ times to find
\begin{equation}\label{eq:inverse5}
\lambda^{N-1}u(N-1)(j)-u(0)(j) = 0.
\end{equation}
Finally, multiplying \eqref{eq:inverse5} by $\lambda$ and adding it to equation \eqref{eq:inverse3} 
we obtain
\begin{equation*}
(\lambda^N-1)u(N-1)(j) = 1,
\end{equation*}
and hence $\lambda^N \neq 1$. Using \eqref{eq:inverse4} we can find the remaining values of $u(m)(j)$:
\[
u(m)(j) = \frac{\lambda^{N-m-1}}{\lambda^N -1}.
\]
This completes the computation of $u$.  We have shown that, for $\lambda e - \delta_{(1,j)}$ to be invertible, 
$\lambda$ must satisfy $\lambda \neq 0$ and $\lambda^N \neq 1$. Given that $\lambda$ meets these 
requirements we found an explicit inverse; therefore $\sigma(\delta_{(1,j)}) = \{0, \lambda : \lambda^N = 1\}.$

By the Riesz representation theorem, a linear functional on $\mathcal{A}$ is given by integration against a 
function $\gamma \in L^\infty(\ZZ/N\ZZ \times \ZZ/d\ZZ),$ which we can also view 
simply as an $N \times d$ matrix. Further, a basic result in the Gelfand theory is that, for a commutative Banach 
algebra with unit, we have $\widehat{x}(\sigma(\mathcal{A})) = \sigma(x)$ (Theorem 1.13 of \cite{foll1995}). 
Combining 
this with our previous calculations, it follows that for a multiplicative linear functional $\gamma$, 
\[
\overline{\gamma(1)(n)} = \int \delta_{(1,n)} \overline{\gamma} = \gamma(\delta_{(1,n)}) \in \sigma(\delta_{(1,n)}).
\]
Since $\gamma$ is multiplicative,
\[
        \overline{\gamma (m)(n)} = \int \delta_{(m,n)} \overline{\gamma} = \int \delta_{(1,n)}^m \overline{\gamma} 
        = \gamma(\delta_{(1,n)}^m) = \gamma(\delta_{(1,n)})^m,
\]
taking the values  $0$ or $\lambda^m$ where $\lambda^N = 1.$
Therefore $\gamma (0)(n)$ is $0$ or $1$, and since 
\[
1 = \gamma(e) = \sum_{k=0}^{d-1} \overline{\gamma (0)(k)},
\] 
we have $\gamma(0)(n) \neq 0$ (and thus $\gamma(1)(n) \neq 0$) for only one $n$. It follows that for this $n$, 
$\gamma(1)(n) = \overline{\lambda}$ where $\lambda^N = 1$. 

We have everything we need to describe $\sigma(\mathcal{A})$. The multiplicative linear functionals on $\mathcal{A}$ 
are $N \times d$ matrices of the form
\[
\gamma_{\lambda,k}(m)(n) =
\begin{cases}
\lambda^{-m} & \text{for } n = k,\\
0 & \text{otherwise,}
\end{cases}
\quad \text{where } \lambda^N = 1, \quad 0 \leq k \leq d-1.
\]
Set $\omega = e^{-2\pi i /N}$. If $\lambda^N = 1$, then $\lambda = \omega^j$ for some $0 \leq j \leq N-1$, 
and we can write $\gamma_{\lambda,k}$ as $\gamma_{j,k}$.  Thus, we can list all the elements of 
$\sigma(\mathcal{A})$ as $\{\gamma_{j,k}\}$, $0 \leq j \leq N-1$, $0 \leq k \leq d-1$, and there are $Nd$ of them. 

Let $s: \ZZ/d\ZZ \rightarrow \ZZ/N\ZZ$ be injective and have the property that for every $n \in \ZZ/d\ZZ$, 
$(s(n),N) = 1$, i.e., the vector-valued {\it DFT} defined by $s$ is invertible. Using $s$, 
we can reorder $\sigma(\Acal)$ 
as follows.  For $0 \leq p \leq N-1$ and $0 \leq q \leq d-1,$ define $\gamma_{p,q}^\prime$ by  
\[\gamma_{p,q}^\prime(m)(n) =
\begin{cases}
\omega^{-p m s(q)} & \text{for } n = q,\\
0 & \text{otherwise}.
\end{cases}
\]
We claim $\{\gamma_{p,q}^\prime\}_{p,q}$ is a reordering of $\{\gamma_{j,k}\}_{j,k}$. To show this, first 
note that $\{\gamma_{p,q}^\prime\}_{p,q} \subseteq \{\gamma_{j,k}\}_{j,k}$. To demonstrate the 
reverse inclusion, 
for each $q \in \ZZ/d\ZZ$ find a multiplicative inverse to $s(q)$ in $\ZZ/N\ZZ$. This may be done 
because $(s(q),N) = 1$ for every $q$. Writing this inverse as $s(q)^{-1},$ it follows that
\[
       \gamma_{js(k)^{-1},k}^\prime = \gamma_{j,k},
\]
and therefore  $\{\gamma_{j,k}\}_{j,k} \subseteq \{\gamma_{p,q}^\prime\}_{p,q}$. 

We summarize all of these calculations as the following theorem.

\begin{thm}[Spectrum and Gelfand transform of $\mathcal{A}$]
\label{thm:specgelfA}
The spectrum, $\sigma(\Acal)$, of $\Acal$ is
identified with $\ZZ/N\ZZ\times\ZZ/d\ZZ$ by means of the mapping 
$\gamma_{p,q}^\prime \leftrightarrow (p,q).$ Under this identification, 
the Gelfand transform, $\widehat{x} \in C(\sigma(\Acal))$,
of $x \in \mathcal{A}$, is the $N \times d$ matrix,
\begin{align*}
        \widehat{x}(p)(q) = \widehat{x}(\gamma_{p,q}^\prime) = \gamma_{p,q}^\prime(x) 
&= \sum_{m=0}^{N-1} x(m)(q)\omega^{pms(q)}\\
&= \sum_{m = 0}^{N-1} x(m)(q)e^{-2\pi i pms(q)/N}.
\end{align*} 
In particular, under the identification, $\gamma_{p,q}^\prime \leftrightarrow (p,q)$, 
the Gelfand transform of $\Acal$ is the vector-valued {\it DFT}. 
\end{thm}

While this shows that
 the transform we have defined is itself not new, it also shows that a classical transform 
 can be redefined in the context of frame theory.  


\section{Formulation of generalized scalar- and vector-valued ambiguity functions}
\label{sec:stftform}

\subsection{Formulation}
\label{sec:form}

Given $u: \ZZ/N\ZZ \rightarrow \CC^d.$ A periodic vector-valued ambiguity function 
$A_p^d(u): \ZZ/N\ZZ \times \ZZ/N\ZZ \rightarrow \CC^d$ 
was defined in \cite{BenDon2008} by observing the following. 
If $d = 1$, then $A_p(u)$ in Equation (\ref{eq:peraf}) can be written as
\begin{align}
\label{eq:discreteambiguityrewrite}
      A_p(u)(m,n) &= \frac{1}{N}\sum_{k=0}^{N-1}\langle u(m+k), u(k)e_{kn}\rangle \nonumber\\
      &= \frac{1}{N}\sum_{k=0}^{N-1}\langle \tau_{-m} u(k), F^{-1}(\tau_n \widehat{u})(k)\rangle,
\end{align}
where $\tau_{-m}$ is the translation operator 
of Definition \ref{defn:transmod} and where $F^{-1}$ is the inverse {\it DFT}
on $\ZZ/N\ZZ$. In particular, we see that $A_p(u)$ has the form of a STFT,
see Example \ref{ex:stftapd}.
This is central to our approach.

If 
$d > 1$, then, motivated by the calculation $(\ref{eq:discreteambiguityrewrite}),$ it turns out that
we can define both a $\CC$-valued ambiguity function  $A_p^1(u)$ and
a $\CC^d$-valued function $ A_p^d(u).$

First, we consider the case of a $\CC$-valued ambiguity function. Inspired by 
\eqref{eq:discreteambiguityrewrite}, 
and for $u: \ZZ/N\ZZ \rightarrow \CC^d$, we wish to {\it construct} a sequence 
$\{x_n\}_{n = 0}^{N-1} \subseteq \CC^d$
and {\it define} a vector multiplication $\ast$  in $\CC^d$ so that the mapping,
$A_p^1(u): \ZZ/N\ZZ \times \ZZ/N\ZZ \rightarrow \CC,$ given by
\begin{equation}
\label{eq:ap1def}
       A_p^1(u)(m,n) = \frac{1}{N}\sum_{k=0}^{N-1}\langle u(m+k), u(k)\ast x_{kn} \rangle
\end{equation}
is  a meaningful ambiguity function. The product, $kn,$ is modular
multiplication in $ \ZZ/N\ZZ.$ In Subsections \ref{sec:apdftframe}
and \ref{sec:crossprod}, we shall see 
that in quite general circumstances, for the proper $\{x_n\}_{n = 0}^{N-1}$ and $\ast,$
Equation (\ref{eq:ap1def}) can be made compatible with that of $A_p(u)$ 
in \eqref{eq:discreteambiguityrewrite}.

Second, we consider the case of a $\CC^d$-valued ambiguity function.
In the context of our definition of  $A_p^1(u),$
we formulate the vector-valued version, $A_p^d(u),$ of the periodic
ambiguity function  as follows. Let
$u: \ZZ/N\ZZ \rightarrow \CC^d$, and define the mapping,
$A_p^d(u): \ZZ/N\ZZ \times \ZZ/N\ZZ \rightarrow \CC^d,$ by
\begin{equation}
\label{eq:apddef}
      A_p^d(u)(m,n) = \frac{1}{N} \sum_{k=0}^{N-1}u(m+k)\ast\overline{u(k)}\ast\overline{x_{kn}},
\end{equation}
where $\{x_n\}_{n =0}^{N-1}$ and $\ast$ must also be constructed and
defined, respectively. In Example
\ref{ex:stftapd}, we shall see that this definition is compatible with that of $A_p(u)$ 
in \eqref{eq:discreteambiguityrewrite}.

To this end of defining $A_p^d(u),$ and motivated by the facts that  $\{e_n\}_{n = 0}^{N-1}$
 is a tight frame for $\CC$ (as noted in Subsection \ref{sec:dftframes}) and $e_m e_n = e_{m+n}$, 
 the following {\it frame multiplication assumptions}
were made in \cite{BenDon2008}. 
\begin{itemize}
\item There is a sequence $X = \{x_n\}_{n = 0}^{N-1} \subseteq \CC^d$ and a multiplication 
$\ast: \CC^d \times \CC^d \rightarrow \CC^d$ such that 
\begin{equation}
\label{eq:ap1rule}
       \forall m, n \in \ZZ/N\ZZ, \quad x_m \ast x_n = x_{m+n};
\end{equation}
\item $X = \{x_n\}_{n = 0}^{N-1}$ is a tight frame for $\CC^d;$
\item  The multiplication $\ast$ is bilinear,
 in particular,
\[
         \left(\sum_{j=0}^{N-1}c_jx_j\right) \ast \left(\sum_{k=0}^{N-1}d_kx_k\right) = \sum_{j=0}^{N-1}
         \sum_{k=0}^{N-1}c_jd_kx_j\ast x_k.
\]

\end{itemize}
There exist tight frames satisfying these 
assumptions, e.g., {\it DFT} frames. We shall characterize 
such tight frames and multiplications in Sections \ref{sec:fm}, \ref{sec:harmgroupframes}, 
and \ref{sec:fmgroupframes}.

A reason we developed our vector-valued {\it DFT} theory of Section
\ref{sec:vvDFT} was to \textit{verify}, not just \textit{motivate},
that $A_p^{d}(u)$ is a STFT in the case $\{x_k\}_{k=0}^{N-1}$ is a {\it DFT}
frame for ${\mathbb C}^d$.
Let $X = \{x_n\}_{n = 0}^{N-1}$ be a {\it DFT} frame for $\CC^d$. We  can leverage 
the relationship between the bilinear product pointwise multiplication and the operation of 
addition on the indices of $X$, i.e., $x_m x_n = x_{m+n}$, to define the periodic vector-valued 
ambiguity function $A_p^{d}(u)$ as in Equation (\ref{eq:apddef}). In this case,
 the {\it DFT} frame is acting as a high dimensional analog to the 
roots of unity $\{\omega_n = e^{2\pi i n/N}\}_{n = 0}^{N-1}$, that appear in the definition of the usual
periodic ambiguity function. 

\begin{example}[Multiplication problem]
\label{ex:multprob}

 Given $u:{\mathbb Z}_N \longrightarrow {\mathbb C}^d$.
If $d=1$ and $x_n=e^{2\pi in /N}$, then Equations (\ref{eq:peraf}) and 
(\ref{eq:discreteambiguityrewrite})
can be written as 
\[
  A_{p}(u)(m,n) = \frac{1}{N}\sum_{k=0}^{N-1}\langle u(m+k),u(k)x_{nk}\rangle.
\]
The {\it multiplication problem} for $A_{p}^1(u)$ is 
to characterize sequences $\{x_k\} \subseteq \CC^d$ and
multiplications $\ast$ so that
$$
A_{p}^{1}(u)(m,n) = \frac{1}{N}\sum_{k=0}^{N-1}\langle u(m+k),u(k)\ast x_{nk}\rangle \in {\mathbb C}
$$
is a meaningful and well-defined {\em ambiguity function.} This
formula is clearly motivated by the STFT. It is for this reason that we made the frame multiplication 
assumptions.

In fact, suppose $\{x_{j}\}_{j=0}^{N-1} \subseteq \CC^d$
satisfies the three frame multiplication assumptions.
If we are given $u,v: \ZZ/N\ZZ \longrightarrow \CC^d$ and $m,n \in \ZZ/N\ZZ,$
then we can make the calculation,
\begin{eqnarray} 
\label{eq:uastv}
u(m)*v(n) &=&
\frac{d}{N}\sum_{j=0}^{N-1} \langle u(m),x_{j} \rangle x_{j}*\frac{d}{N}
\sum_{s=0}^{N-1} <v(n),x_{s}>x_{s} \\
&=& \frac{d^2}{N^2}\sum_{j=0}^{N-1}\sum_{s=0}^{N-1}
<u(m),x_{j}><v(n),x_{s}>x_{j}*x_{s} \nonumber \\
&=&\frac{d^2}{N^2}\sum_{j=0}^{N-1}\sum_{s=0}^{N-1}
\langle u(m),x_{j}\rangle \langle v(n),x_{s}\rangle x_{j+s}. \nonumber
\end{eqnarray}
This allows us to formulate $A_{p}^1(u),$ as written in Equation (\ref{eq:ap1def}),
see Subsection \ref{sec:apdftframe}.

\end{example}


\subsection{$A_{1}^{d}(u)$ and $A_{p}^{d}(u)$ for DFT frames}
\label{sec:apdftframe}

\begin{example}[STFT formulation of $A_p^{1}(u)$]
\label{ex:ap1dftframe}
Given $u,v : {\mathbb Z}/N{\mathbb Z} \rightarrow \mathbb C^d$, and let
$X = \{x_k\}_{k=0}^{N-1} \subseteq \CC^d$ be a {\it DFT} frame for $\mathbb C^d$.
Suppose $\ast $ denotes  pointwise (coordinatewise) multiplication
times a factor of $\sqrt{d}$. Then, the frame multiplication assumptions are satisfied.
To see this, and without loss of generality, choose the first $d$ columns of
the $N \times N$ {\it DFT} matrix, and let $r$ designate a fixed column.
Then, we can verify the first of the frame multiplication assumptions by
the following calculation, where the first step is a consequence of Equation \eqref{eq:uastv}:
\[
x_m\ast x_n(r) = \frac{d^2}{N^2}\sum_{j=0}^{N-1}\sum_{s=0}^{N-1}
\langle x_{m},x_{j}\rangle \langle x_{n},x_{s}\rangle x_{j+s}(r). 
\]
\[
= \frac{1}{N^2\sqrt{d}}\sum_{j=0}^{N-1}\sum_{s=0}^{N-1}\sum_{t=0}^{d-1}
\sum_{k=0}^{d-1}x_{(m-j)t}x_{(n-s)k}x_{(j+s)r} 
\]
\[
=  \frac{1}{N^2\sqrt{d}}\sum_{t=0}^{d-1}\sum_{k=0}^{d-1}x_{mt + nk}\sum_{j=0}^{N-1}
x_{(r-t)j}\sum_{s=0}^{N-1}x_{(r-k)s} 
\]
\[
= \frac{1}{N^2\sqrt{d}}\sum_{t=0}^{d-1}\sum_{k=0}^{d-1}x_{mt + nk}N\delta(r-t)N
\delta(r-k) 
\]
\[
=\frac{x_{(m +n)r}}{\sqrt{d}} = x_{m+n}(r). 
\]
The second and third frame multiplication assumptions follow since $X$ is  a
{\it DFT} frame and by a straightforward calculation (already used in
Equation \eqref{eq:uastv}), respectively.

Thus, in this case, $A_{p}^{1}(u)$ is well-defined for
$u:\mathbb \ZZ/N\ZZ \rightarrow \mathbb C^d$ by
Equation \eqref{eq:ap1exists} since its right side exists:
\begin{equation}
\label{eq:ap1exists}
A_{p}^{1}(u)(m,n) = \frac{1}{N}\sum_{k=0}^{N-1}\langle u(m+k),u(k) \ast x_{nk}\rangle 
\end{equation}
\[
 = \frac{1}{N}\sum_{k=0}^{N-1}\left<u(m+k),\frac{d}{N}\sum_{j=0}^{N-1}
<u(k),x_{j}>x_{j}*x_{nk}\right> 
\]
\[
 = \frac{d}{N^2}\sum_{k=0}^{N-1}\sum_{j=0}^{N-1}\langle x_{j},u(k)\rangle \langle u(m+k),x_{j+nk}\rangle.
\]
\end{example}

\begin{example}[STFT formulation of $A_p^{d}(u)$]
\label{ex:stftapd}
Given $u,v : {\mathbb Z}/N{\mathbb Z} \rightarrow \mathbb C^d$, and let
$X = \{x_k\}_{k=0}^{N-1} \subseteq \CC^d $ be a {\it DFT} frame for $\mathbb C^d$.
Suppose $\ast $ denotes  pointwise (coordinatewise) multiplication
with a factor of $\sqrt{d}$. Then, the frame multiplication assumptions are satisfied.
Utilizing the modulation functions, $e^j,$
defined in Definition \ref{defn:transmod},  we compute the right side of Equation (\ref{eq:apddef}) 
to obtain
\begin{equation}
\label{eq:apd2}
A_p^d(u)(m,n) = \frac{1}{N} \sum_{k=0}^{N-1}\tau_{-m}u(k)\overline{u(k)e^n(k)}.
\end{equation}
Furthermore, the modulation and translation properties of the vector-valued {\it DFT} 
allow us to write Equation (\ref{eq:apd2}) as
\[
     A_p^{d}(u)(m,n) =
    \frac{1}{N}\sum_{k=0}^{N-1}(\tau_{m}u(k)) \ast\overline{F^{-1}(\tau_{n}\hat{u})(k)};
\]
and, notationally, we write the right side as the generalized inner product,
\[
   \frac{1}{N}\sum_{k=0}^{N-1}\{\tau_{m}u(k),F^{-1}(\tau_{n}\hat{u})(k)\},
\]
where $\{u,v\} = u  \overline{v}$ is coordinatewise multiplication for $u, v \in \CC^d.$
Because of the form of Equation (\ref{eq:apddef}),  we reiterate that $A_p^{d}(u)$ is compatible 
with the 
point of view of defining
a  vector-valued ambiguity function in the context of the STFT.
\end{example}


\subsection{A generalization of the frame multiplication assumptions}
\label{sec:formala1def}
In the previous {\it DFT} examples, $\ast$ is intrinsically related to 
modular addition defined on the indices of the frame elements,
viz., $x_m \ast x_n = x_{m+n}$.
Suppose we are given $X$ and $\ast,$ that satisfy the frame multiplication
assumptions. It is not pre-ordained that the operation on the indices of the frame
$X$, induced by the bilinear vector 
multiplication, be addition mod $N$, as is the case for {\it DFT} frames. We are interested in finding 
tight frames whose behavior is similar
to that of {\it DFT} frames and whose index sets are Abelian groups, non-Abelian groups, or more 
general non-group sets and operations. 

Hence, and being formulaic, we could have ${x}_{m} \ast {{x}_n} = {x}_{m \bullet n}$ for some function 
$\bullet:{\mathbb Z}/N{\mathbb Z} \times  {\mathbb Z}/N{\mathbb Z} \rightarrow {\mathbb Z}/N{\mathbb Z},$ 
and, thereby, we could use non-{\it DFT} frames or even non-{\it FUNTFs} for ${\mathbb C}^d$.
Further, $\bullet$ could be defined on index sets, that are more general than ${\mathbb Z}/N{\mathbb Z}.$
Thus, a particular case 
could have the setting of bilinear mappings
of frames for Hilbert spaces that are indexed by groups.

For the purpose of Subsection \ref{sec:crossprod}, we continue to consider
the setting of ${\mathbb Z}/N{\mathbb Z}$ and ${\mathbb C}^d$, but replace
the first frame multiplication assumption, Equation \eqref{eq:ap1rule}, by the formula,
\begin{equation}
\label{eq:astbullet}
     \forall m, n \in \ZZ/N\ZZ, \quad x_m \ast x_n = x_{m \bullet n},
\end{equation}
where $X = \{x_k\}_{k=0}^{N-1}$ is still a tight frame for  $\CC^d $ and where 
$\ast$ continues to be bilinear.

The formula, Equation \eqref{eq:astbullet}, not only hints at generalization by the 
cross-product example of Subsection \ref{sec:crossprod}, but is the 
formal basis of the theory of frame multiplication in Sections
\ref{sec:fm}, \ref{sec:harmgroupframes}, 
and \ref{sec:fmgroupframes}.


\subsection{Frame multiplication assumptions for cross product frames}
\label{sec:crossprod}

Let $\ast :\mathbb C^3 \times \mathbb C^3 \rightarrow
{\mathbb C}^3$ be the
cross product on $\mathbb C^3$ and let $\{i,j,k\}$ be the standard basis,
e.g., $i=(1,0,0)\in\CC^3.$
Therefore, we have
\begin{equation}
\label{eq:crossops}
i\ast j=k, j\ast i=-k, k\ast i=j, i\ast k = -j, j\ast k=i, k\ast j=-i,
\end{equation}
\[
i\ast i=j\ast j =k\ast k=0. 
\]
 The union of 
two tight frames and the zero vector is a tight frame, so if we let
$X = \{x_n\}_{n = 0}^6,$ where $x_0=0,x_1 =i, x_2 =j, x_3=k, x_4=-i, 
x_5=-j, x_6=-k,$ then it is straightforward to check that  $X$ is a tight frame for $\CC^3$ 
with frame bound $2.$

The index operation
corresponding to the frame multiplication is 
\begin{equation}
\label{eq:zz7zz}
   \bullet :\ZZ/7\ZZ  \times \ZZ/7\ZZ \longrightarrow \ZZ/7\ZZ,
\end{equation}
where $\bullet$ is the non-Abelian, non-group operation defined by the following table:

\begin{center}
\begin{tabular}{ccccc}
$1\bullet 2 =3$, & $1\bullet 3 =5$ & $1\bullet 4=0$,&  $1\bullet 5=6$, & $1\bullet 6=2$,\\ 
$2\bullet 1 =6$, & $2\bullet 3=1$, & $2\bullet 4=3$,& $2\bullet 5=0$,& $2\bullet 6 = 4$,\\
$3\bullet 1 =2$, & $3\bullet 2=4$, & $3\bullet 4=5$,& $3\bullet 5=1$,& $3\bullet 6 = 0$,\\
\multicolumn{2}{c}{$n\bullet n=0$,}&\multicolumn{2}{c} {$n\bullet 0=0\bullet n =0$.}
\end{tabular}
\end{center}
We have chosen this definition of $\bullet$ for the following reasons.
As we saw in Example \ref{ex:multprob}, the three frame multiplication assumptions are essential
for defining a meaningful ambiguity function. In Subsection \ref{sec:form}, these
assumptions were based on the formula,
$x_m \ast x_n = x_{m+n},$
used in Equation \eqref{eq:ap1rule}. However, in order to generalize
this point of view, we shall consider the
formula, $x_m \ast x_n = x_{m \bullet n},$ as indicated in Subsection \ref{sec:formala1def}.
provided the corresponding three frame multiplication assumptions can be verified.
In fact, for this cross product example, it is easily checked that
 the frame multiplication assumptions of Equation \eqref{eq:ap1rule} are 
valid when $+$ is replaced by the $\bullet$ operation defined above in \eqref{eq:zz7zz}
and the corresponding table.

Consequently,we can write the
cross product as
\begin{equation}
\label{eq:crossproduct}
      \forall u,v \in \CC^3,\quad   u \times v = u\ast v= \frac{1}{2^2}\sum_{s=1}^{6}
      \sum_{t=1}^{6}\langle u,x_{s}\rangle \langle v,x_{t}\rangle x_{s\bullet t}.
\end{equation}
\[
= \frac{1}{4} \sum_{n = 1}^6 \left(\sum_{j\bullet k = n}\inner{u}{x_j}\inner{v}{x_k} \right)x_n.
\]


One possible application of the above is that, given frame representations for $u, v \in \CC^3,$
Equation \eqref{eq:crossproduct} allows us to compute the frame representation of $u \times v$ without 
the process of going back and forth between the frame representations and
 their standard orthogonal representations. 

There are five non-isomorphic groups of order 8: the Abelians ($\ZZ/8\ZZ, \ZZ/4\ZZ \times \ZZ/2\ZZ,
\ZZ/2\ZZ \times \ZZ/2\ZZ \times\ZZ/2\ZZ $), the dihedral, cf. Example \ref{ex:d3gmatrix}, and the
quaternion group $Q = \{ \pm 1, \pm i, \pm j, \pm k \}.$ The unit of $Q$ is $1,$ the products $ij,$ etc.
are the cross product as in Equation \eqref{eq:crossops}, and $ii = jj = kk =1.$ Clearly, 
$Q$ is
closely related to $X = \{x_n\}_{n = 0}^6,$


\section[]{Frame multiplication}
\label{sec:fm}

We now define the notion of a frame multiplication, that is connected with a
bilinear product on the frame elements, and we analyze its properties.

\begin{defn}[Frame multiplication]Let $X = \{x_j\}_{j \in J}$ be a frame for a 
$d-$dimensional Hilbert space $H$
over ${\mathbb F}$, and let $\bullet: J \times J \rightarrow J$ be a binary operation. 
The mapping $\bullet$ is a {\it frame multiplication} for $X$ or, by abuse of language, a frame 
multiplication for $H$, if it extends to a bilinear product $\ast$ on all of $H$, i.e., if there exists a bilinear 
product $\ast : H \times H \rightarrow H$ such that 
\[
\forall j, k \in J, \quad x_j \ast x_k = x_{j \bullet k}.
\]
If  $(G, \bullet)$ is a group, where $G=J$ and $\bullet$ is a frame multiplication for $X$, then 
we shall also say that $G$ {\it defines a frame multiplication for} $X.$

To fix ideas, we shall only deal with frame multiplication for finite dimensional Hilbert spaces,
but our theory clearly extends, and many of the results are valid for infinite dimensional
Hilbert spaces.
\end{defn}

Let $X = \{x_j\}_{j \in J}$ be a frame for a $d-$dimensional Hilbert space $H.$
By definition, a binary operation $\bullet: J \times J \rightarrow J$ is a frame 
multiplication for $X$ when it extends to a bilinear product by bilinearity
to the entire space $H$. Conversely, if there is a bilinear product 
$\ast: H \times H \rightarrow H$ which agrees with $\bullet$ on $X,$ i.e., 
$x_j \ast x_k = x_{j\bullet k}$, then it must 
be the unique extension given by bilinearity since $X$ spans $H$. Therefore, $\bullet$ defines a frame multiplication 
for $X$ if and only if for every $x = \sum_i a_i x_i \in H$ and $y = \sum_i b_i x_i \in H$, 
\begin{equation}
\label{eq:framemultiplication2}
      x \ast y = \sum_{i \in J}\sum_{j \in J} a_i b_j  x_{i \bullet j}
\end{equation}
is defined and independent of the frame representations used for $x$ and $y$. 

\begin{rem}
Whether or not a particular binary operation is a frame multiplication depends not just on the elements of the 
frame but on the indexing of the frame. For clarity of definitions and later theorems, we make no attempt to 
define a notion of frame multiplication for multi-sets of vectors that is independent of the index set. 

A distinction that must be kept in mind is that $\bullet$ is a set operation on the indices of a frame 
while $\ast$ is a bilinear 
vector product defined on all of $H$. 
\end{rem}

We shall investigate the interplay between bilinear vector products on $H$, frames for $H$ indexed by $J$, 
and binary operations on $J$. For example, if we fix a binary operation $\bullet$ on $J$, then 
for what sort of 
frames indexed by $J$ do we obtain a frame multiplication? Conversely, if we fix a 
frame $X = \{x_j\}_{j\in J}$ for $H,$
then what sort of binary operations on $J$ are frame multiplications for $H$?

\begin{prop}
\label{prop:fmexistence}
Let $X = \{x_j\}_{j \in J}$ be a frame for a 
$d$-dimensional Hilbert space $H,$ and let $\bullet: J \times J \rightarrow J$ be a 
binary operation. Then, $\bullet$ is a frame multiplication for $X$ if and only if
\begin{equation}
\label{eq:fmexistence}
        \forall \{a_i\}_{i \in J} \subseteq \FF \; {\rm and} \; \forall j \in J, \quad \sum_{i \in J}a_i x_i = 0 \,
        {\rm implies} \, 
        \sum_{i \in J}a_i x_{i\bullet j}= 0 \, {\rm and } \, \sum_{i \in J}a_i x_{j\bullet i} = 0.
\end{equation}
\begin{proof}
Suppose $\ast$ is the bilinear product defined by $\bullet$ and $\{a_i\}_{i \in J}$ is a sequence of scalars. 
If $\sum_{i \in J} a_i x_i =0$, then 
\begin{equation*}
     \sum_{i \in J} a_i x_{i \bullet j} = \sum_{i \in J} a_i \left(x_i \ast x_j \right) =
      \left(\sum_{i \in J}a_i x_i \right) \ast x_j = 0 \ast x_j = 0.
\end{equation*}
Similarly, by multiplying on the left by $x_j$, we see that $\sum_{i \in J} a_i x_{j \bullet i} = 0.$

For the converse, suppose that statement \eqref{eq:fmexistence} holds and 
$x = \sum_i a_i x_i = \sum_i c_i x_i$, 
$y = \sum_j b_j x_j = \sum_j d_j x_j$ $\in H$. By \eqref{eq:fmexistence}, we have
\begin{equation}
\label{eq:fmcondition1}
         \forall j \in J, \quad \sum_{i \in J} (a_i - c_i) x_{i \bullet j} = 0
\end{equation}
and
\begin{equation}
\label{eq:fmcondition2}
        \forall i \in J, \quad \sum_{j \in J} (b_j - d_j) x_{i \bullet j} = 0.
\end{equation}
Therefore, using \eqref{eq:fmcondition2} and \eqref{eq:fmcondition1}, we obtain
\[
\sum_{i \in J} \sum_{j \in J} a_i b_j x_{i \bullet j}  = \sum_{i \in J} a_i \sum_{j \in J} b_j x_{i \bullet j} 
= \sum_{i \in J} a_i \sum_{j \in J} d_j x_{i \bullet j}
\]
\[ 
= \sum_{j \in J} d_j \sum_{i \in J} a_i x_{i \bullet j}
= \sum_{j \in J} d_j \sum_{i \in J} c_i x_{i \bullet j} 
= \sum_{i \in J} \sum_{j \in J} c_i d_j x_{i \bullet j},
\]
and, hence,  $\ast$ is well-defined by \eqref{eq:framemultiplication2}.
\end{proof}
\end{prop}

\begin{defn}[Similarity]
Frames $X = \{x_j\}_{j \in J}$ and $Y = \{y_j\}_{j \in J}$ for a 
$d$-dimensional Hilbert space $H$ are {\it similar} if there exists 
an invertible linear operator $A \in \Lcal(H)$ (see Appendix \ref{sec:unirep})
such that
\[
       \forall j \in J, \quad Ax_j = y_j.
\]
\end{defn}

\begin{prop}
\label{prop:fminvertibleequiv}
Suppose $X = \{x_j\}_{j \in J}$ and $Y = \{y_j\}_{j \in J}$ are frames for 
a $d$-dimensional Hilbert space $H,$ and that $X$ is similar to $Y$. 
Then, a binary operation $\bullet : J \times J \rightarrow J$ is a frame multiplication for $X$ 
if and only if it is a 
frame multiplication for $Y$.
\begin{proof}
Because $A^{-1}y_j = x_j$ and $A^{-1}$ is also an invertible operator, we need only prove one 
direction of the proposition. Suppose $\bullet$ is a frame multiplication for $X$ 
and that $\sum_i a_i y_i = 0$. We have  
\[
0 = \sum_{i \in J} a_i y_i = \sum_{i \in J} a_i A x_i = A\left( \sum_{i \in J} a_i x_i \right),
\]
and since $A$ is invertible it follows that $\sum_i a_i x_i = 0$. By 
Proposition \ref{prop:fmexistence}, and because
$\bullet$ is a frame multiplication for $X$, we can assert that
\[
   \forall j \in J, \quad \sum_{i \in J}a_i x_{i\bullet j}= 0 \text{ and } \sum_{i \in J}a_i x_{j\bullet i} = 0.
\]
Applying $A$ to both of these equations yields:
\[
     \forall j \in J, \quad \sum_{i \in J}a_i y_{i\bullet j}= 0 \text{ and } \sum_{i \in J}a_i y_{j\bullet i} = 0.
\]
Therefore, by Proposition \ref{prop:fmexistence}, $\bullet$ is a frame multiplication for $Y$. 
\end{proof}
\end{prop}

\begin{defn}[Multiplications of a frame]
Let $X = \{x_j\}_{j \in J}$ be a frame for a $d$-dimensional Hilbert space $H$. 
The {\it multiplications} of $X$ are defined and denoted by
\begin{equation*}
     {\rm mult}(X) = \{{\rm frame}\, {\rm multiplications} \; \bullet : J \times J \rightarrow J \;
     {\rm for} \, X\}.
\end{equation*} 
$ {\rm mult}(X)$ can be all functions (for example when $X$ is a basis), empty, or somewhere 
in-between.
\end{defn}

\begin{example}
\label{ex:nofm}
Let $\alpha, \beta > 0$, $\alpha \neq \beta$, and $\alpha + \beta <1$. Define $X_{\alpha, \beta} = 
\{x_1 = (1, 0)^t, x_2 = (0,1)^t, x_3 = \left(\alpha, \beta \right)^t\}$. Notationally, the superscript $t$
denotes the {\it transpose} of a vector.
Then, $X_{\alpha, \beta}$ is a frame for 
$\CC^2$ and ${\rm mult}(X_{\alpha,\beta}) = \emptyset$. 
A straightforward way to prove that ${\rm mult}(X_{\alpha,\beta}) = \emptyset$ is to show that there are no bilinear operations 
on $\CC^2$ which leave $X_{\alpha,\beta}$ invariant. Suppose $\ast$ were such a bilinear operation. 
We shall obtain a contradiction.

First, we have the 
linear relation $x_3 = \alpha x_1 + \beta x_2$. Hence, by the bilinearity of $\ast$,
\begin{equation}
\label{eq:existenceexample}
x_1 \ast x_3 = \alpha x_1 \ast x_1 + \beta x_1 \ast x_2.
\end{equation}
Second, $\norm{x_1}_2 = \norm{x_2}_2 = 1$ and $\norm{x_3}_2 < 1$, where the inequality follows 
from the facts that $\norm{x_3}_2 = (\alpha^2 + \beta^2)^{1/2}$ and 
\[
    0 < \alpha^2 + \beta^2 = (\alpha + \beta)^2 - 2\alpha\beta <  (\alpha +\beta)^2 < 1.
\]
By the properties of $\alpha,\,\beta$,
and using Equation \eqref{eq:existenceexample}, we have that
\begin{equation}
\label{eq:existenceexineq}
     \exists\,m,\,n \; {\rm such}\,{\rm that} \; \norm{x_1 \ast x_3}_2 =
    \norm{\alpha x_1\ast x_1 +  \beta x_1 \ast x_2}_2 
\end{equation}
\[
      = \norm{\alpha x_m + \beta x_n}_2  \leq \alpha \norm{x_m}_2 + \beta \norm{x_n}_2 < 1.
\]
Thus, since $\ast$ leaves $X_{\alpha,\beta}$ invariant, 
we obtain that $x_1 \ast x_3 = x_3$ by \eqref{eq:existenceexineq}.
 Furthermore, substituting $x_3$ for 
$x_1 \ast x_3$ in Equation \eqref{eq:existenceexample} 
and using the assumption that $\alpha \neq \beta$, yield $x_1 \ast x_1 = x_1$ and $x_1 \ast x_2 = x_2$. 
Performing the analogous calculation
on $x_3 \ast x_2,$ in place of $x_1 \ast x_3$ above, shows that $x_2 \ast x_2 = x_2$ 
and $x_1 \ast x_2 = x_1,$ the desired contradiction.
\end{example}

\setlength{\unitlength}{1cm}
\begin{figure}[h]
\begin{center}
\frame{\begin{picture}(6,6)
\linethickness{1pt}
\put(1,1){\vector(1,0){4}}
\put(1,1){\vector(0,1){4}}
\put(1,1){\vector(2,1){2}}
\put(4.5,1.25){$x_1$}
\put(0.5,5){$x_2$}
\put(2.5,2.25){$x_3$}
\end{picture}}
\end{center}
\caption[A frame with no frame multiplications]{The frame $X_{\alpha,\beta}$ from 
Example \ref{ex:nofm} for $\alpha = 1/2$ and $\beta = 1/4$. This frame has no frame multiplications.}
\end{figure}
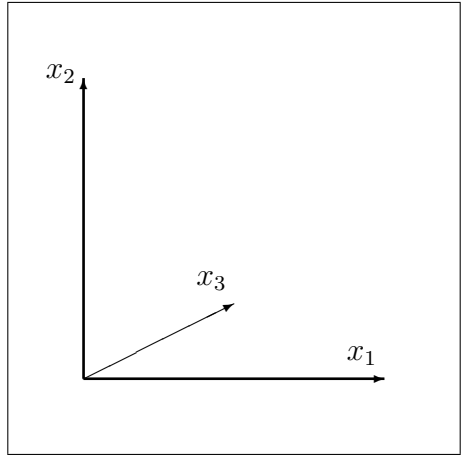

Of particular interest, Proposition \ref{prop:fminvertibleequiv} tells us that the 
canonical dual frame $\{S^{-1}x_j\}_{j \in J}$ 
and the canonical tight frame $\{S^{-1/2}x_j\}_{j \in J}$ share the same frame
multiplications as the original frame $X$. 
Because of this, we shall focus our attention on tight frames. An invertible element 
$V \in \Lcal(H)$ mapping an 
$A$-tight frame $X = \{x_j\}$ (frame constant $A$) to an $A^\prime$-tight frame $Y = \{y_j\}$, as in 
Proposition \ref{prop:fminvertibleequiv}, is a positive multiple 
of some $U \in \Ucal(H),$ the space of unitary operators
on $H,$ see Appendix \ref{sec:unirep}. Indeed, we have
\[
     A \norm{V^\ast x}^2 = \sum_{j \in J} \left| \inner{V^\ast x}{x_j}\right|^2 =  
     \sum_{j \in J} \left| \inner{x}{V x_j}\right|^2 = \sum_{j \in J} \left| \inner{x}{y_j}\right|^2 = A^\prime \norm{x}^2.
\]
This leads us to a notion of equivalence for tight frames that sounds stronger than similarity but is actually 
just the restriction of similarity to the class of tight frames.

\begin{defn}[Equivalence of tight frames]
Tight frames $X = \{x_j\}_{j \in J}$ and $Y = \{y_j\}_{j \in J}$ for a 
separable Hilbert space $H$ are {\it unitarily equivalent} 
if there is $U \in \Ucal(H)$ and a positive constant $c$ such that
\[
         \forall j \in J, \quad x_j = cUy_j.
\]
\end{defn}

Whenever we speak of equivalence classes for tight frames we shall mean under unitary equivalence. 
For finite frames unitary equivalence can be stated in terms of Gramians: 

\begin{prop}[\cite{ValWal2005}]
\label{prop:gram}
Let $H$ be a $d$-dimensional Hilbert space, and let
$X = (x_1, \ldots, x_N)$ and $Y = (y_1, \ldots, y_N)$ be sequences of vectors. Suppose 
${\rm span}(X) = H,$ and so $X$ is a frame for $H.$ There exists $U \in \Ucal(H)$ 
such that $Ux_i = y_i$,for every $i = 1, \ldots, n$,  if and only if
\begin{equation*}
      \forall i, j \in \{1, \ldots, N\} \quad \inner{x_i}{x_j} = \inner{y_i}{y_j}, 
\end{equation*}
i.e., the Gram matrices of $X$ and $Y$ are equal.
\end{prop}

Thus, from Proposition \ref{prop:gram}, tight frames $X$ and $Y$ are unitarily equivalent if and only 
if one of their Gramians is a scaled version of the other. In the case where both $X$ and $Y$ are 
equivalent Parseval 
frames their Gramians are equal.

We are using Deguang Han and David Larson's \cite{HanLar2000} 
definition of 
similarity and unitary equivalence. In particular, the ordering of the 
frame, and 
not just the unordered set of frame elements, is important. This choice is in concert with the way 
in which we have defined frame multiplication, i.e., with a fixed index for our frame. Also, we have made 
no attempt to define equivalence for frames indexed by different sets. This can be done, 
and results can then 
proven about the correspondence of frame multiplications between similar or equivalent frames under 
this new definition. However, allowing frames with two different index sets of the same cardinality to be 
considered similar only obfuscates our results.

\begin{thm}[Multiplications of equivalent frames]
\label{thm:unitaryequiv}
Let $X = \{x_j\}_{j \in J}$ and $Y = \{y_j\}_{j \in J}$ be finite tight frames for 
a $d$-dimensional Hilbert space $H$. If $X$ is unitarily equivalent 
to $Y$, then 
${\rm mult}(X) = {\rm mult}(Y)$.
\begin{proof}
Since $X$ and $Y$ are unitarily equivalent they are similar. Therefore, by 
Lemma \ref{prop:fminvertibleequiv}, 
$\bullet: J \times J \rightarrow J$ defines a frame multiplication on $X$ if and only if it defines a frame multiplication 
on $Y$, that is, ${\rm mult}(X) = {\rm mult}(Y)$.
\end{proof}
\end{thm}

The converse of Theorem \ref{thm:unitaryequiv} is not valid. The multiplications of a tight frame 
provide a coarser equivalence relation than unitary equivalence. In fact, 
as Example \ref{ex:uncount} demonstrates, we may have 
uncountably many equivalence classes of tight frames, that have  the same multiplications. 

\begin{example}
\label{ex:uncount}
Let $\{\alpha_i\}_{i = 1,2}$ and $\{\beta_i\}_{i = 1,2}$ be real numbers such that 
$\alpha_1 > \beta_1 > \alpha_2 > \beta_2 > 0$,  $\alpha_1 + \beta_1<1$, and $\alpha_2 + \beta_2< 1$. 
Define $X_{\alpha_1, \beta_1}$ and $X_{\alpha_2,\beta_2}$ as in Example \ref{ex:nofm}. 
Then ${\rm multi}(X_{\alpha_1,\beta_1}) = 
{\rm multi}(X_{\alpha_2,\beta_2}) = \emptyset$. It can be easily shown, by checking the six cases of where to 
map $(1,0)^t$ and $(0,1)^t$, that there is no invertible operator $A$ such that $AX_{\alpha_1,\beta_1} = 
X_{\alpha_2,\beta_2}$ as sets. Therefore, there are no $c > 0$ and $U \in \Ucal(\RR^2)$ such that $cU$ maps 
between the canonical tight frames $S_1^{-1/2}X_{\alpha_1,\beta_1}$ and $S_2^{-1/2}X_{\alpha_2,\beta_2}$ 
(for any reordering of the elements) and $S_1^{-1/2}X_{\alpha_1,\beta_1}$ 
and $S_2^{-1/2}F_{\alpha_2,\beta_2},$ 
are not unitarily equivalent. Hence, there are uncountably many equivalence classes of tight frames, that
have
the same empty set of frame multiplications.
\end{example}

In contrast to Example \ref{ex:uncount}, we shall see in 
Section \ref{sec:fmgroupframes} that if a tight frame has a 
particular frame multiplication in terms of a group operation, then it belongs to one of only finitely 
many equivalence classes of tight frames, that share the same group operation as a 
frame multiplication. With this goal, we close this subsection with a characterization of bases in terms of 
their multiplications, once we exclude the degenerate one case where 
one can have a frame consisting of a single repeated vector).

\begin{prop}
Let $X = \{x_j\}_{j \in J}$ be a finite frame for a 
$d$-dimensional Hilbert space $H,$ and suppose $dim(H) > 1.$
If ${\rm multi}(X) = \{ {\rm all}\,{\rm functions}  \, \bullet: J \times J \rightarrow J\}$, then $X$ 
is a basis for $H.$
If, in addition, $X$ is a tight, respectively, Parseval frame for $H,$ then $X$ is an orthogonal,
respectively, orthonormal basis for $H.$
\begin{proof}
Suppose that $\sum_i a_i x_i = 0$, $j_0 \in J$, and $x_{j_1}, x_{j_2} \in X$ are linearly 
independent. Let $\bullet: J \times J \rightarrow J$ be the function sending all products to $j_2$ 
except that
\[
     \forall j \in J, \quad j_0 \bullet j = j_1.
\]
By assumption, $\bullet \in {\rm multi}(X)$. Therefore, by Proposition \ref{prop:fmexistence},
we have
\[
     \forall j \in J, \quad 0 = \sum_{i \in J} a_i x_{i \bullet j} = a_{j_0} x_{j_1} + \sum_{i \neq j_0}a_i x_{j_2}.
\]
Since $x_{j_1}$ and $x_{j_2}$ are linearly independent, $a_{j_0} = 0$, and since $j_0$ was 
arbitrary, $X$ is a linearly independent set. The last statement of the proposition follows from 
the elementary fact that a basis, that satisfies Parseval's identity or a scaled version
of it, is an orthogonal set.
\end{proof}
\end{prop}


\section{Harmonic frames and group frames}
\label{sec:harmgroupframes}


\subsection{Background}
The central part of our theory in Section \ref{sec:fmgroupframes}
depends on the well-established setting of harmonic frames
and group frames. We review that
material here. We shall see that harmonic frames are group frames.

These are two of several 
related classes of frames and codes that have been the object of recent study. B\"{o}lcskei and 
Eldar \cite{BolEld2003} (2003) 
define {\it geometrically uniform} frames as the orbit of a generating vector under an Abelian group of unitary 
matrices. A signal space code was called {\it geometrically uniform} by Forney \cite{forn1991} (1991) or a 
{\it group code} by Slepian \cite{slep1968} (1968) 
if its symmetry group (a group of isometries) acts transitively. {\it Harmonic frames} are projections of the rows 
or columns of the character table (Fourier matrix) of an Abelian group. Georg Zimmermann \cite{zimm2001}
and G{\"o}tz Pfander [unpublished] independently introduced and provided substantial properties
of  harmonic frames at Bommerholz in 1999,
cf. \cite{HanLar2000} (2000), \cite{CasKov2003} (2003), \cite{HeaStr2003} (2003), \cite{ValWal2005} (2005),  
\cite{hirn2010} (2010), \cite{ValWal2010} (2010), \cite{ChiWal2011} (2011).  
In \cite{ValWal2005} it was shown 
that harmonic frames and geometrically uniform tight frames 
are equivalent and can be characterized by their Gramian. 

Let us both expand and focus on of the definition of a harmonic frame in the
previous paragraph.
It is a well known result that the rows and columns of the character table of an Abelian group 
are orthogonal. This fact combined with the direction of Naimark's theorem,
Theorem \ref{thm:naimark}, asserting that the orthogonal projection of an orthogonal basis 
is a tight frame, motivates considering the class of
equal-norm frames $X$ of $N$ vectors for a $d$-dimensional Hilbert space $H$ 
that arise from the 
character table of an Abelian group, i.e., equal norm frames given by the columns 
of a submatrix obtained by 
taking $d$ rows of the character table of an Abelian group of order $N$.

With more precision, we
state the following definition.

\begin{defn}[Harmonic frame for an Abelian group]
Let $(G, \bullet)= \{g_1, \ldots, g_N\}$ be a finite Abelian group with dual group 
$\{\gamma_1, \ldots, \gamma_N\}.$ The $N \times N$ matrix with $(j,k)$ entry
$\gamma_k(g_j)$ is the {\it character table} of $G.$ Let $K \subseteq \{1, \ldots, N\},$
where $|K| = d \leq N,$ and with columns indexed by $k_1, \ldots, k_d.$ 
Let $U \in \Ucal(\CC^d).$ The {\it harmonic frame} $X = X_{G,K,U}$ for  $\CC^d$ is
\[
    X = \{U(\gamma_{k_1}(g_j), \ldots, \gamma_{k_d}(g_j)) : j = 1, \ldots, N\}.
\] 
\end{defn}

Given $G, K, $ and $U = I.$ Then, $X$ is the {\it DFT - FUNTF} on $G$ for 
$\CC^d.$ In this case, if $G = \ZZ/N\ZZ,$ then $X$ is the usual
{\it DFT - FUNTF} for $\CC^d.$

A fundamental characterization of harmonic frames is due to Vale and Waldron
\cite{ValWal2005} (2005), and the intricate evaluation of the number of
harmonic frames of prime order is due to Hirn \cite{hirn2010} (2010.



\subsection{Group frames}
\label{sec:groupframes}

We begin with the first definition of a group frame from 
Han \cite{han-1997a} (1997),  where the associated representation $\pi$ is called a 
{\it frame representation}, also see  \cite{HanLar2000} by Han and Larson (2000).

\begin{defn}
\label{def:groupframe1}
Let $(G, \bullet)$ be a finite group. A finite frame $X$ for a $d$-dimensional
Hilbert space $H$ is a {\it group frame} if there exists 
$\pi : G \rightarrow \mathcal{U}(H)$, a unitary representation of $G$, and $x \in H$ such that
\[
      X = \{\pi(g)x\}_{g \in G}.
\]
\end{defn}

If $X$ is a group frame, then $X$ is generated by the orbit of any one of its elements 
under the 
action of $G$, and if $X$ contains $N$ unique vectors, then each element of $X$ is 
repeated $|G|/N$ times. 
If $e$ is the group identity, then we fix an ``identity" element $x_e$ of $X,$  and write 
$X = \{x_g\}_{g \in G},$ where $x_g = \pi(g)x_e$. From this we see that group frames are 
frames for which there exists an indexing by the group $G$ such that
\[
      \pi(g)x_h = \pi(g)\pi(h)x_e = \pi(g \bullet h)x_e = x_{g \bullet h}.
\]
This leads to a second, essentially equivalent, definition of a group frame for a frame already 
indexed by $G$. This is the definition used by Vale and Waldron in \cite{ValWal2008}.

\begin{defn}[Group frame]
\label{def:groupframe2}
Let $(G, \bullet)$ be a finite group, and let $H$
be a $d$-dimensional Hilbert space. A finite tight frame $X = \{x_g\}_{g \in G}$ for $H$ is a 
{\it group frame} if there exists 
\[
       \pi : G \longrightarrow \mathcal{U}(H),
\]
a unitary representation of $G,$ such that
\[
   \forall g, h \in G, \quad \pi(g)\,x_h = x_{g \bullet h}.
\]
\end{defn}

\begin{example}
The difference between Definitions \ref{def:groupframe1} and \ref{def:groupframe2}
is that in Definition \ref{def:groupframe2} we begin with a frame as a sequence indexed by $G,$ 
and then ask 
whether a particular type of representation exists. In the first definition we began with 
only a multi-set of vectors 
and asked whether an indexing exists such that the second definition holds. For example, 
let $G = \ZZ/4\ZZ = 
(\{0,1,2,3\},+)$ and consider the frame $X = \{x_0 = 1, x_1 = -1, x_2
= i, x_3 = -i\}$ for $\CC$. $X$ would be a group frame under Definition \ref{def:groupframe1}, 
because there are 
two one-dimensional representations of $G$ that generate $X.$ This is clear from the Fourier 
matrix of $\ZZ/4\ZZ$. 
However, it would not qualify as a group frame under Definition \ref{def:groupframe2}, because 
the representation $\pi$ would have to 
satisfy $\pi(1)x_0 = x_1,$ i.e., $\pi(1)1 = -1.$
There is one one-dimensional representation 
of $\ZZ/4\ZZ$ which satisfies this, but it does not generate $X$. Indeed, it is defined by 
$\pi(0) = 1, \pi(1) = -1, \pi(2) = 1, \pi(3) = -1$. 
\end{example}

In keeping with our view that a frame is a sequence with a fixed index
set, we shall use the second 
definition.

\begin{example}
Harmonic frames are group frames.
\end{example}

Vale and Waldron noted in \cite{ValWal2008} that if $X = \{x_g\}_{g \in G}$ is a 
group frame, then its Gramian matrix $(G_{g,h}) = (\inner{x_h}{x_g})$ has a special form:
\begin{equation}
\label{eq:gframegramian}
      \inner{x_h}{x_g} = \inner{\pi(h)x_g}{\pi(g)x_g} = \inner{x_g}{\pi(h^{-1} \bullet g)x_g},
\end{equation}
i.e., the $g$-$h$-entry is a function of $h^{-1} \bullet g$. 

\begin{defn}[G-matrix]
Let $G$ be a finite group. A matrix $A = (a_{g,h})_{g, h \in G}$ is called a {\it G-matrix} 
if there exists a function $\nu: G \rightarrow \CC$ such that
\begin{equation*}
         \forall g, h \in G, \quad a_{g,h} = \nu(h^{-1} \bullet g).
\end{equation*}
\end{defn}

Vale and Waldron \cite{ValWal2008} were then able to prove 
essentially the following 
theorem using an 
argument that hints at a connection to frame multiplication. We include a version of their 
proof and highlight the connections to our theory.

\begin{thm}
\label{thm:gmatrix}
Let $G$ be a finite group. A frame $X = \{x_g\}_{g \in G}$ for a 
$d$-dimensional Hilbert space $H$ is a 
group frame if and only if its Gramian is a G-matrix.

\begin{proof}
If $X$ is a group frame, then Equation \eqref{eq:gframegramian} implies its Gramian is the 
G-matrix defined by $\nu(g) = \inner{x_e}{\pi(g)x_e}$.

For the converse, suppose the Gramian of $X$ is a G-matrix. Let $S$ be the frame operator, 
and let $\widetilde{x}_g = S^{-1}x_g$ be the canonical dual frame elements. 
Each $x \in H$ has the frame decomposition
\begin{equation}
\label{eq:gmatrixproof1}
           x = \sum_{h \in G} \langle x, \widetilde{x}_h\rangle x_h.
\end{equation}
For each $g \in G$, define a linear operator $U_g:H \rightarrow H$ by
\begin{equation*}
       \forall x \in H, \quad U_g(x) = \sum_{h \in G} \langle x, \widetilde{x}_h\rangle x_{g \bullet h}.
\end{equation*}
Since the Gramian of $X$ is a G-matrix, we have
\begin{equation}
\label{eq:gmatrixproof3}
           \forall g, h, k \in G, \quad \inner{x_{g \bullet h}}{x_{g \bullet k}} 
           = \nu((g \bullet h)^{-1}g \bullet k) 
           = \nu(h^{-1} \bullet k) = \inner{x_{h}}{x_{k}}.
\end{equation}
The following calculation shows that $U_g$ is unitary, and the
calculation itself follows from \eqref{eq:gmatrixproof1} and \eqref{eq:gmatrixproof3}:
\[
        \inner{U_g(x)}{U_g(y)} 
        = \inner{\sum_{h \in G} \langle x, \widetilde{x}_h\rangle x_{g \bullet h}}
        {\sum_{k \in G} \langle y, \widetilde{x}_k\rangle x_{g \bullet k}} 
\]
\[
    = \sum_{h \in G} \sum_{k \in G} \langle x, \widetilde{x}_h\rangle 
    \overline{\langle y, \widetilde{x}_k\rangle} \inner{x_{g \bullet h}}{x_{g \bullet k}}
     = \sum_{h \in G} \sum_{k \in G} \langle x, \widetilde{x}_h\rangle 
      \overline{\langle y, \widetilde{x}_k\rangle} \inner{x_{h}}{x_{k}}
\]
\[
          = \inner{\sum_{h \in G} \langle x, \widetilde{x}_h\rangle x_{h}}
          {\sum_{k \in G} \langle y, \widetilde{x}_k\rangle x_{k}} 
          = \inner{x}{y}.
\]
Also, for every $h, k \in G,$ we compute
\[
\inner{U_g(x_h)-x_{g \bullet h}}{x_{g \bullet k}} = \inner{U_g(x_h)}{x_{g \bullet k}}
-\inner{x_{g \bullet h}}{x_{g \bullet k}}
\]
\[
=\inner{\sum_{m \in G}\inner{x_h}{\widetilde{x}_m}x_{g \bullet m}}{x_{g \bullet k}}
-\inner{x_{g \bullet h}}{x_{g \bullet k}}
\]
\[
=\sum_{m \in G}\inner{x_h}{\widetilde{x}_m}\inner{x_{g \bullet m}}{x_{g \bullet k}}
-\inner{x_{g \bullet h}}{x_{g \bullet k}}
=\sum_{m \in G}\inner{x_h}{\widetilde{x}_m}\inner{x_{m}}{x_{k}}-\inner{x_{h}}{x_{k}}
\]
\[
=\inner{\sum_{m \in G}\inner{x_h}{\widetilde{x}_m}x_{m}}{x_{k}}-\inner{x_{h}}{x_{k}}\\
=\inner{x_h}{x_k}-\inner{x_h}{x_k}
= 0.
\]
Letting $k$ vary over all of $G,$ it follows that $U_g(x_h) = x_{g \bullet h}$.
This implies that $\pi : g \mapsto U_g$ is a unitary representation, since
\[
        \forall g_1, g_2, h \in G, \quad U_{g_1 \bullet g_2} x_h = x_{g_1 \bullet g_2 \bullet h} 
         = U_{g_1}x_{g_2 \bullet h} = U_{g_1} U_{g_2} x_h
\]
and since $\{x_h\}_{h \in G}$ spans $H$. Hence, $\pi$ is a unitary representation of $G$ 
for which
$\pi(g)x_h = x_{g \bullet h}$, i.e., $X$ is a group frame for $H$.
\end{proof}
\end{thm}

\begin{rem}
The operators $U_g, g \in G, $ defined in the proof of 
Theorem \ref{thm:gmatrix} are essentially frame multiplication on the left by $x_g,$
but there 
may not exist a bilinear product on all of $H$ which agrees with or properly joins the sequence 
$\{U_g\}_{g \in G}$. We shall prove in Proposition \ref{prop:unitary} that when these operators 
do arise from a 
frame multiplication defined by $G$, then they are unitary when the Gramian is a G-matrix. 
In fact, we shall see in Section \ref{sec:fmgroupframes} that, if $G$ is an {\it Abelian} group 
and if the Gramian of $X = \{x_g\}_{g \in G}$ is a 
G-matrix, or by Theorem \ref{thm:gmatrix} if $X$ is a group frame, then $G$ defines a 
frame multiplication for $X.$ 
\end{rem}

\begin{example}
If $G$ a cyclic group, a G-matrix is a circulant matrix. To illustrate this, we consider 
$G = \ZZ/4\ZZ = (\{0,1,2,3\}, +)$ with the natural ordering. Then all G-matrices, 
corresponding to this choice of $G,$ are of the form
\[\left(
\begin{array}{cccc}
\nu(0) & \nu(3) & \nu(2) & \nu(1) \\
\nu(1) & \nu(0) & \nu(3) & \nu(2) \\
\nu(2) & \nu(1) & \nu(0) & \nu(3) \\
\nu(3) & \nu(2) & \nu(1) & \nu(0) \\
\end{array}
\right)\]
for some $\nu : G \rightarrow \CC$, 
and this is a $4\times 4$ circulant matrix.
\end{example}

\begin{example}
\label{ex:d3gmatrix}
For a non-circulant example of a G-matrix, let $G = D_3$, the dihedral group of order $6$. If we use the presentation, 
\[
        D_3 = <r, s : r^3 = e, s^2 = e, rs = sr^2>,
\]
and order the elements $e, r, r^2, s, sr, sr^2$, then every G-matrix has the form
\[\bordermatrix{\text{}&e & r & r^2 & s & sr & sr^2 \cr
 e     &\nu(e) & \nu(r^2) & \nu(r) & \nu(s) & \nu(sr) & \nu(sr^2) \cr
 r      &\nu(r) & \nu(e) & \nu(r^2) & \nu(sr) & \nu(sr^2) & \nu(s) \cr
 r^2 &\nu(r^2) & \nu(r) & \nu(e) & \nu(sr^2) & \nu(s) & \nu(sr) \cr
 s     &\nu(s) & \nu(sr) & \nu(sr^2) & \nu(e) & \nu(r^2) & \nu(r) \cr
 sr    &\nu(sr) & \nu(sr^2) & \nu(s) & \nu(r) & \nu(e) & \nu(r^2) \cr
 sr^2&\nu(sr^2) & \nu(s) & \nu(sr) & \nu(r^2) & \nu(r) & \nu(e) }\]
for some $\nu: D_3 \rightarrow \CC$.
\end{example}

\section{Frame multiplication for group frames}
\label{sec:fmgroupframes}


\subsection{Frame multiplication defined by groups}
\label{sec:groupops}

We now deal with the special case of frame multiplications defined by binary operations 
$\bullet: J \times J \rightarrow J$ that are group operations, i.e., when $J = G$ is a group and 
$\bullet$ is the group operation. Recall that if $X = \{x_g\}_{g \in G}$ 
is a frame for a Hilbert space $H$ and the group operation of $G$ is a frame multiplication 
for $X$, then we say that 
$G$ defines a frame multiplication for $X$. 


We state and prove Proposition \ref{prop:unitary} in some generality to
illustrate the basic idea and its breadth. We use it 
to prove Theorem \ref{thm:gframes}.
 
\begin{prop}
\label{prop:unitary}
Let $(G, \bullet)$ be a countable group, and let $H$ be a separable Hilbert space.
Assume $X = \{x_g\}_{g \in G}$ is a tight frame for $H.$ 
If $G$ defines a frame multiplication for $X$ 
with continuous extension $\ast$ to all of $H,$
then the functions $L_g: H \rightarrow H,$ defined by
\[
       L_g(x) = x_g \ast x,
\]
and $R_g: H \rightarrow H,$ defined by
\[
     R_g(x) = x \ast x_g,
\]
are unitary linear operators for every $g \in G$.
\begin{proof}
Let $x \in H$, $g \in G$, and $A$ be the frame constant for $X$. Linearity and continuity of $L_g$ follow  
from the bilinearity and continuity of $\ast$. To show that $L_g$ is unitary, we first compute
\[
A \norm{L_g^\ast(x)}^2 = \sum_{h \in G}\left|\inner{L_g^\ast(x)}{x_h}\right|^2
=\sum_{h \in G}\left|\inner{x}{L_g(x_h)}\right|^2
\]
\[
=\sum_{h \in G}\left|\inner{x}{x_g \ast x_h}\right|^2 
=\sum_{h \in G}\left|\inner{x}{x_{gh}}\right|^2 
=\sum_{h \in G}\left|\inner{x}{x_h}\right|^2 
=A \norm{x}^2.
\]
Therefore, $L_g^\ast$ is an isometry. If $H$ is finite dimensional, this is equivalent to $L_g^\ast$ 
and $L_g$ being 
unitary. 

For the infinite dimensional case, we also need that $L_g$ is an isometry, this being one of 
the equivalent characterizations of unitary operators. 

To prove that $L_g$ is an isometry, we 
first show it is invertible and that
$L_g^{-1} = L_{g^{-1}}$. To this end, we begin by defining
\[
       D = \left\{\sum_h a_h x_h : |\{a_h: a_h \neq 0\}| < \infty\right\},
\] i.e., $D$ is 
the set of all finite linear combinations of 
frame elements from $X$. It follows from the frame reconstruction formula that $D$ is dense in $H$. 
Now, for any $g \in G$, $L_g$ maps $D$ onto $D$, and for every $x = \sum_{h \in G} a_h x_h \in D,$ 
we compute
\[
         L_{g^{-1}} L_g (x) = L_{g^{-1}} L_g \left( \sum_{h \in G}a_h x_h \right) 
\]
\[
         = L_{g^{-1}} \left( \sum_{h \in G}a_h x_{g \bullet h} \right)\\
         = \sum_{h \in G}a_h x_h = x.
\]
In short, $L_{g^{-1}}L_g$ is linear, bounded, and is the identity on a dense subspace of $H$. Therefore,
$L_{g^{-1}}L_g$ 
is the identity on all of $H$.

We can now verify that $L_g$ is an isometry. From general operator theory, we have the equalities
\[
      \norm{L_g^{-1}}_{op} = \norm{L_{g^{-1}}}_{op} = \norm{L_{g^{-1}}^\ast}_{op} = 1,
\]
and
\[
\norm{L_g}_{op} = \norm{L_g^\ast}_{op} = 1.
\]
Invoking these and the definition of the operator norm, we obtain
\[
      \norm{L_g (x)} \leq \norm{x} \quad \text{and} \quad \norm{x} = 
        \norm{L_g^{-1} L_g (x)} \leq \norm{L_g^{-1}}_{op}\norm{L_g (x)} = \norm{L_g (x)}.
\]
Therefore, $\norm{L_g(x)} = \norm{x},$ the desired isometry. 

The same calculations
prove that $R_g$ is unitary.
\end{proof}
\end{prop}

In contrast to the generality of Proposition \ref{prop:unitary}, we 
next give a specific example providing direction that led to our main results in 
Subsection \ref{sec:abelfm}.

\begin{example}
 Let $X = \{x_k\}_{k=0}^{N-1} \subseteq \CC^d$ be a linearly 
dependent frame for $\CC^d,$ and so $N > d.$ Suppose 
$\ast : \CC^d \times \CC^d \rightarrow \CC^d$ is a bilinear 
product such that $x_m \ast x_n = x_{m+n}$, i.e., $\ZZ/N\ZZ$ defines a 
frame multiplication for $X$. 
By linear dependence, there exists a sequence $\{a_k\}_{k=0}^{N-1} \subseteq \CC$ 
of coefficients, not all zero, 
such that
\[\sum_{k=0}^{N-1} a_k x_k = 0.\]
Multiplying on the left by $x_m$ and utilizing the aforementioned properties of $\ast$ yield
\begin{equation}
\label{eq:gframeex}
     \forall m \in \ZZ/N\ZZ, \quad    0 = x_m \ast \left(\sum_{k=0}^{N-1} a_k x_k\right) = 
     \sum_{k=0}^{N-1} a_k \left(x_m \ast x_k \right)= 
         \sum_{k=0}^{N-1} a_k x_{m+k}.
\end{equation}
It is convenient to rewrite \eqref{eq:gframeex} with the index on the coefficients varying with $m:$
\begin{equation}
\label{eq:gframeex2}
           \forall m \in \ZZ/N\ZZ, \quad \sum_{k=0}^{N-1} a_{k-m} x_k = 0. 
\end{equation}
Let $\mathbf{a} = 
(a_k)_{k=0}^{N-1}$, let $A$ be the $N \times N$ circulant matrix generated by the vector $\mathbf{a}$
and with eigenvalues $\lambda_j, j = 0, \ldots, N-1, $
and let $X$ be the $N \times d$ matrix with vectors $x_k$ as its rows. In symbols,
\[
A = \left(
\begin{array}{ccccc}
a_0 & a_1 & a_2 & \ldots & a_{N-1} \\
a_{N-1} & a_0 & a_1 & \ldots & a_{N-2} \\
\vdots & \vdots  & \vdots & \ddots  & \vdots \\
a_1 & a_2 & a_3 & \ldots & a_0 \\
\end{array}
\right)
\quad \mbox{and} \quad 
X = \left(
\begin{array}{c}
x_0 \\
x_1 \\
\vdots \\
x_{N-1} \\
\end{array}
\right).
\]
 In matrix form, Equation \eqref{eq:gframeex2} is 
\[
       AX = 0.
\]
Thus, the columns of $X$ are in the nullspace of the circulant matrix $A.$
A consequence of this and of the fact that the {\it DFT} diagonalizes circulant matrices 
is that the columns of $X$ 
are linear combinations of some subset of at least $d$ (the rank of $X$ is $d$) 
columns of the {\it DFT} matrix.
Further, 
if $\omega_j = e^{2\pi i j/N}$, then 
\[
     a_0 + a_{N-1} \omega_j + a_{N-2} \omega_j^2 + \ldots + a_1 \omega_j^{N-1} = \lambda_j,
\] 
is zero for at least $d$ choices of $j \in \{0, 1, \ldots, N-1\}.$ Hence, assuming 
that $\ZZ/N\ZZ$ defines a frame multiplication for a frame $X$ for $\CC^d,$  
we obtain a 
 condition involving the {\it DFT}. 
\end{example}

\subsection{Abelian frame multiplications}
\label{sec:abelfm}

\begin{thm}[Abelian frame multiplications for group frames]
\label{thm:gframes}
Let $(G, \bullet)$ be a finite Abelian group, and let $H$ be a $d$-dimensional 
Hilbert space. Assume that $X = \{x_g\}_{g \in G}$ is a tight frame for $H$. 
$G$ defines a frame multiplication for $X$ if and only if $X$ is a group frame.

\begin{proof}
{\it i.} Suppose $G$ defines a frame multiplication for $X$ and the bilinear product given on 
$H$ is $\ast$. For each $g \in G$ define an operator $U_g: H \rightarrow H$ by the formula
\[
U_g(x) = x_g \ast x.
\]
By Proposition \ref{prop:unitary}, $\{U_g\}_{g \in G}$ is a family of unitary operators on $H$. 
Define the mapping $\pi:g \mapsto U_g$. $\pi$ is a unitary representation of $G$ because 
\[
       U_g U_h x_k = U_g (x_h \ast x_k) 
       = U_g(x_{h \bullet k}) = x_g \ast x_{h \bullet k} = x_{g \bullet h \bullet k} 
       = U_{g \bullet h}x_k,
\]
and since $\{x_k\}_{k \in G}$ spans $H.$
Further, we have $\pi(g)x_h = x_{g \bullet h},$ thereby proving $X$ is a group frame.

{\it ii.} Conversely, suppose $X = \{x_g\}_{g\in G}$ is a group frame. Then, there exists a unitary 
representation $\pi$ of $G$ such that $\pi(g)x_h = x_{g \bullet h}$. It follows from the facts,
$\pi(g)$ is unitary and $G$ is Abelian, that
\begin{equation}
\label{eq:fmcond}
       \forall g, h_1,h_2 \in G, \quad \inner{x_{h_1}}{x_{h_2}} =  
       \inner{\pi(g)x_{h_1}}{\pi(g)x_{h_2}} =  
       \inner{x_{g \bullet h_1}}{x_{g \bullet h_2}} =  \inner{x_{h_1 \bullet g}}{x_{h_2 \bullet g}}.
\end{equation}

{\it iii.} If $\sum_{g \in G} a_g x_g = 0$, then for any $j, k \in G$ we have
\begin{align*}
0 &= \inner{\sum_{g \in G} a_g x_g}{x_j}
= \sum_{g \in G}a_g \inner{x_g}{x_j}\\
&= \sum_{g \in G} a_g \inner{x_{g \bullet k}}{x_{j\bullet k}}
= \inner{\sum_{g \in G} a_g x_{g \bullet k}}{x_{j\bullet k}}.
\end{align*}
Allowing $j$ to vary over all of $G$ shows that $\sum_{g \in G}a_g x_{g \bullet k} = 0$. 
Similarly, we can use the fact that $\inner{x_g}{x_j} = \inner{x_{k\bullet g}}{x_{k\bullet j}}$ to show 
$\sum_{g \in G}a_g x_{k\bullet g} = 0$. Hence, by Proposition \ref{prop:fmexistence}, $\bullet$ is a frame 
multiplication for $X$.
\end{proof}
\end{thm}


If $(G, \bullet)$ is a finite Abelian group and $H = \CC^d,$ 
then we can describe the form of frame multiplications 
defined by $G$ in the following way.

\begin{thm}[Abelian frame multiplications for harmonic frames]
\label{thm:harmonicframemultiplications}
Let $(G, \bullet)$ be a finite Abelian group.
Assume that $X = \{x_g\}_{g \in G}$ is a tight frame for 
$\CC^d.$
If $G$ defines a frame multiplication for $X$, then $X$ is unitarily equivalent to a harmonic frame, 
and there exists $U \in  \mathcal{U}(\CC^d)$ and $c > 0$ such that
\begin{equation}\label{eq:harmonicframemultiplication}
cU\left(x_g \ast x_h\right) = cU\left(x_g\right)cU\left(x_h\right),
\end{equation}
where the product on the right is vector pointwise multiplication and $\ast$ is the frame multiplication 
defined by $(G, \bullet),$ i.e., $x_g \ast x_h = x_{g \bullet h}$.
\end{thm}

\begin{proof}
{\it i.} For each $g \in G$ define an operator $U_g: \CC^d \rightarrow \CC^d$ by the formula
\[
      U_g(x) = x_g \ast x.
\]
By Theorem \ref{thm:gframes}, $\{U_g\}_{g \in G}$ is an Abelian group of unitary operators,
that generates
\[
        X = \{U_g(x_e): g \in G\},
\]
where $e$ is the unit of $G.$
Furthermore, since the $U_g$ are unitary, we have
\[
        \forall g \in G, \quad \norm{x_e}_2 = \norm{U_g(x_e)}_2 = \norm{x_g}_2,
\]
which shows that $X$ is equal-norm. 

{\it ii.} For the next step we use a technique found in 
the proof of Theorem 5.4 of \cite{ValWal2005}. A commuting family of diagonalizable operators, 
such as $\{U_g\}_{g \in G}$, is 
simultaneously diagonalizable, i.e., there is a unitary operator $V$ for which
\[
        \forall g \in G, \quad D_g = VU_gV^\ast
\]
is a diagonal matrix, see \cite{HofKun1971} Theorem 6.5.8, 
cf. \cite{HorJoh1990} Theorem 2.5.5.

This is a also consequence of Schur's lemma 
and Maschke's theorem, see Appendix \ref{sec:unirep}.
Since $\{U_g\}_{g \in G}$ is an Abelian group of operators, all the 
invariant subspaces are one dimensional, and so, orthogonally decomposing $\CC^d$ 
into the invariant subspaces of $\{U_g\}_{g \in G},$ simultaneously diagonalizes the operators 
$U_g$. The operators $D_g$ are unitary, and consequently, they have diagonal entries of 
modulus $1$. 

{\it iii.} Define a new frame, $Y,$ generated by the diagonal operators $D_g$,  as
\[
        Y = \{D_gy : g \in G\}, \; \text{where} \; y = V(x_e).
\]
Since $V^\ast D_gV = U_g$, we have
\[
      X = \{U_g(x_e): g \in G\} = V^\ast Y,
\]
or
\[
          VX = Y.
\]
Let $(D_gy)_j$ be the $j$-th component of the vector $D_gy$. Form the $d \times |G|$ matrix 
with columns the elements of $Y$, i.e., if we write $G =\{g_1, \ldots, g_N\}$, then we form
\begin{equation}
\label{eq:psimatrix}
\left(
\begin{array}{ccc}
(D_{g_1}y)_0 & \ldots & (D_{g_N}y)_0\\
(D_{g_1}y)_1 & \ldots & (D_{g_N}y)_1\\
\vdots & \ddots & \vdots\\
(D_{g_1}y)_{d-1} & \ldots & (D_{g_N}y)_{d-1}\\
\end{array}
\right).
\end{equation}
Since $Y$ is the image of $X$ under $V$, it is an equal-norm tight frame, 
and the synthesis operator matrix \eqref{eq:psimatrix} 
has orthogonal rows of equal length. We compute the 
norm of row $j$ to be
\[
      \left(\sum_{g} |(D_g)(y)_j|^2\right)^{1/2} = \sqrt{|G|}|(y)_j|.
\]
Therefore, the components of $y$ have equal modulus, and, so,  If we let $W^\ast$ be the 
diagonal matrix with the entries of $y$ on its diagonal, then there exists $c > 0$ such that 
$cW^\ast$ is a unitary matrix. 

{\it iv.} Now, we have
\[
         X = \frac{1}{c} U^\ast \{D_g\mbbu : g \in G\}, \quad \text{where } \mbbu = 
         (1, 1, \ldots, 1)^t \text{ and } U^\ast = cV^\ast W^\ast \text{ is unitary}.
\]
It is important to note that we have more than just the equality of sets of vectors as stated above.
In fact, the $g$'s on both sides coincide under the transformation, i.e.,
\begin{align*}
         \frac{1}{c} U^\ast (D_g\mbbu) &= V^\ast W^\ast D_g (\mbbu)= V^\ast D_g (y)\\
             &= U_g V^\ast (y) = U_g (x_e) = x_g.
\end{align*}

Thus, we have found a unitary operator $U$ and $c > 0$ such that $cUx_g = D_g\mbbu$. 

{\it v.} It remains to show that $\{D_g\mbbu : g \in G\}$ is a harmonic frame and that the product $\ast$ 
behaves as claimed. Proving $\{D_g\mbbu : g \in G\}$ is harmonic amounts to showing, for 
$j = 0, 1, \ldots, d-1$, that the mapping,
\[
       \gamma_j: G \rightarrow \CC,
\]
defined by
\[
           \gamma_j(g) = (D_g\mbbu)_j = (D_g)_{jj}
\]
is a character of the group $G$. This follows since
\[
       \forall j = 0,\ldots,d-1, \quad \gamma_j(gh) = (D_{gh})_{jj} = (D_g)_{jj}(D_h)_{jj} = 
       \gamma_j(g)\gamma_j(h).
\]
and $|(D_g)_{jj}| = 1.$

Finally, because $cU(x_g) = D_g\mbbu,$ we can compute
\[
cU(x_g \ast x_h) = cU(x_{gh})
\]
\[
= D_{gh}\mbbu\\
= (D_g\mbbu)(D_h\mbbu)\\
= cU(x_g) cU(x_h).
\qedhere
\]
\end{proof}


\begin{rem}

Strictly speaking, we could have canceled $c$ from both sides of 
Equation (\ref{eq:harmonicframemultiplication}). 
We left them in place because, as we saw in the proof, $cU$ maps the tight frame $X$ to a 
harmonic frame. Therefore, it is made clearer what \eqref{eq:harmonicframemultiplication} means 
when each $c$ is in place, i.e., performing the frame multiplication defined by $G$ and then 
mapping to the harmonic frame is the same as first mapping to the harmonic frame
and then multiplying pointwise.
\end{rem}

In much of our discussion motivating this material, we assumed there was a bilinear product on $\CC^d$ and a 
frame $X$ such that $x_{m} \ast x_{n} = x_{m+n}$, i.e., our underlying group was $\ZZ/N\ZZ$. 
By strengthening our assumptions on $X$ to be a tight frame, we can apply Theorem \ref{thm:gframes} 
to show that $X$ is a group frame for the Abelian group $\ZZ/N\ZZ$, and furthermore, by 
Theorem \ref{thm:harmonicframemultiplications}, $X$ is unitarily equivalent to a {\it DFT} frame, i.e., a 
harmonic frame with $G = \ZZ/N\ZZ.$ Therefore, we have the following corollary.

\begin{cor}
Let $X = \{x_n\}_{n \in \ZZ/N\ZZ} \subseteq \CC^d$ be a tight frame for $\CC^d.$ If $\ZZ/N\ZZ$ 
defines a frame 
multiplication for $X$, then $X$ is unitarily equivalent to a {\it DFT} frame.
\end{cor}

\begin{example}
Consider the group $\ZZ/4\ZZ$, and let 
\[
     X = \left\{x_{0} = \left(\begin{array}{c}1+i\\1-i\end{array}\right), x_{1} = 
     \left(\begin{array}{c}0\\2\end{array}\right), x_{2} = 
     \left(\begin{array}{c}1-i\\1+i\end{array}\right), x_{3} = 
     \left(\begin{array}{c}2\\0\end{array}\right)\right\}.
\]
$X = \{x_g\}_{g \in \ZZ/4\ZZ}$ is a tight frame for $\CC^2$, and the Gramian of $X$ is
\[
G = \left(\begin{array}{cccc}
4 & 2+2i & 0 & 2-2i\\
2-2i & 4 & 2+2i & 0\\
0 & 2-2i & 4 & 2+2i\\
2+2i & 0 & 2-2i & 4\\
\end{array}
\right).
\]
It is straightforward to check that $G$ is a G-matrix for $\ZZ/4\ZZ$, and therefore, 
by Theorems \ref{thm:gmatrix} 
and \ref{thm:gframes}, $\ZZ/4\ZZ$ defines a frame multiplication for $X$. 
Hence, by Theorem \ref{thm:harmonicframemultiplications}, there exists a 
unitary matrix $U$ and positive constant $c$ 
such that $cUX$ is a harmonic frame. Indeed, if we let
\[
      c = \frac{1}{\sqrt{2}}, \quad U = \frac{1}{\sqrt{2}}\left(\begin{array}{cc}1 & 1\\-i & i\end{array}\right),
\]
then
\[
         Y = cUX = \left\{y_{\overline{0}} = \left(\begin{array}{c}1\\1\end{array}\right), y_{\overline{1}} = 
\left(\begin{array}{c}1\\i\end{array}\right), y_{\overline{2}} = \left(\begin{array}{c}1\\-1\end{array}\right), 
y_{\overline{3}} = \left(\begin{array}{c}1\\-i\end{array}\right)\right\}
\]
is a harmonic frame, and
\[
        \forall g,h \in \ZZ/4\ZZ, \quad cU(x_{gh}) = cU(x_g)cU(x_h).
\]
\end{example}


\section{Uncertainty Principles}
\label{sec:uncertaintyprinciples}


\subsection{Background}

Uncertainty principle inequalities abound in harmonic analysis, e.g., see \cite{robe1929}, 
\cite{debr1967}, \cite{fari1978}, \cite{FefPho1981}, \cite{CowPri1983}, \cite{feff1983}, 
\cite{bour1988}, \cite{DonSta1989}, \cite{stri1989}, \cite{DemCovTho1991}, \cite{bene1994}, 
\cite{HavJor1994}, \cite{cohe1995}, \cite{FolSit1997}, \cite{groc2001}, \cite{BenDel2016}.
The classical Heisenberg uncertainty principle is 
deeply rooted in quantum mechanics , see \cite{heis1927}, \cite{weyl1950}, \cite{vonn1955}, \cite{gabo1946}. 
The classical mathematical uncertainty principle inequality was first stated and proved in the setting of 
$L^2(\mathbb R)$ 
in 1924 by Norbert Wiener at a Gottingen seminar \cite{barn1970}, 
also see \cite{kenn1927}. This is Theorem \ref{thm:heisenberg}. 

\begin{thm}[Heisenberg uncertainty principle inequality]
\label{thm:heisenberg}
If $f \in L^2(\RR)$ and $x_0,\gamma_0 \in \mathbb{R}$, then
\begin{equation}
\label{eq:hup}
        \norm{f}_2^2 \leq 4 \pi         
         \left(\int (x-x_0)^2 |f(x)|^2 \,dx\right)^{1/2}
         \left(\int (\gamma - \gamma_0)^2|\widehat{f}(\gamma)|^2\,d\gamma  \right)^{1/2},
\end{equation}
and there is equality if and only if 
\[
 f(x)= C e^{2\pi i x \gamma_0 } e^{-s(x-x_0)^2},
\]
for some $C \in \mathbb{C}$ and $s>0$. 
\end{thm}

The proof of the basic inequality, (\ref{eq:hup}), in Theorem \ref{thm:heisenberg} is a consequence 
of the following calculation for 
$(x_0, \gamma_0)= (0,0)$ and for $f \in \Scal(\mathbb R)$, the Schwartz class of infinitely 
differentiable rapidly decreasing functions defined on $\mathbb R$. 

\begin{equation}
\label{classical proof}
   ||f||_2^4 =  \left( \int_{\mathbb R} x|f(x)^2|' dx \right)^2 
        \leq \left(\int_{\mathbb R}|x| |f(x)^2|' dx \right)^2
\end{equation}
\[
      \leq  \, 4 \left(\int_{\mathbb R} |x \overline{f(x)}f'(x)| dx\right)^2\\
\]
\[
    \leq  \, 4 ||x f(x)||_2^2 ||f'(x)||_2^2 = 16 \pi^2 ||x f(x)||_2^2 ||\gamma \widehat{f}(\gamma)||_2^2.
\]
Integration by parts gives the first equality and the Plancherel theorem gives the second equality;
the third inequality of (\ref{classical proof}) is a consequence of H\"older's inequality,
cf. the proof of (\ref{eq:hup}) in Subsection \ref{sec:uncertainty}.
For more complete proofs, see, for example,  \cite{weyl1950}, \cite{bene1989}, 
\cite{FolSit1997}, \cite{groc2001}.
Integration by parts and Plancherel's theorem can be generalized significantly by means
of Hardy inequalities and weighted Fourier transform norm inequalities, respectively, to yield
extensive weighted generalizations of Theorem \ref{thm:heisenberg}, 
see \cite{BenDel2016} for a technical outline of this program by one of the authors 
in his long collaboration with Hans Heinig and Raymond Johnson.


\subsection{The classical uncertainty principle and self-adjoint operators }
\label{sec:uncertainty}

Let $A$ and $B$ be linear operators on a Hilbert space $H.$ The {\it commutator } $[A,B]$
of $A$ and $B$ is defined as
\[
     [A,B] = AB - BA.
\]
Let $D(A)$ denote the domain of $A.$
The {\it expectation} or {\it expected value} of a self-adjoint operator $A$ in a {\it state} $x \in H$ 
is defined by the expression
\[
       E_x(A) = \langle A \rangle = \langle Ax, x \rangle;
\]
and, since $A$ is self-adjoint, we have 
\[
          \langle A^2 \rangle = \langle Ax, Ax \rangle = \norm{Ax}^2.
\]
The {\it variance} of a self-adjoint operator $A$ at $x \in D(A^2)$ is defined by the
expression
\[
       \Delta ^2_x (A) = E_x(A^2) - \{E_x(A)\}^2.
\]
$\langle A \rangle$ and $\langle A^2 \rangle$ depend on a state $x \in H,$
but traditionally $x$ is often not explicitly mentioned.

We begin with the following Hilbert space uncertainty principle inequality.

\begin{thm}[\cite{bene1994}, Theorem 7.2]
\label{thm:hilbertup1} 
Let $A$, $B$ be self-adjoint operators on a complex Hilbert space $H$ ($A$ and $B$ need not be continuous). 
If 
$$
          x \in D(A^2) \cap D(B^2) \cap D(i[A,B])
$$ 
and $\norm{x}\leq 1$, then
\begin{equation}
\label{eq:hspace1}
\{E_x(i[A,B])\}^2 \leq 4 \Delta_x^2(A) \Delta_x^2(B).
\end{equation} 
\end{thm}

\begin{proof}
By self-adjointness, we first compute
\begin{equation}\label{7.33}
\begin{aligned}
E_x (i[A,B])= & i\left( \langle Bx,Ax \rangle - \langle Ax, Bx \rangle \right) \\
=& 2 \textrm{Im} \langle Ax, Bx \rangle.
\end{aligned}
\end{equation}
Also note that $D(A^2) \subseteq D(A)$. 

Since $||x||\leq 1$ and $\langle Ax, x \rangle$, $\langle Bx,x \rangle \in \mathbb{R}$ by self-adjointness, we have
\begin{align} \label{7.34}
||(B+iA)x||^2 - |\langle (B+iA)x,x \rangle |^2 \geq 0
\end{align}
and
\begin{align} 
\label{7.35}
|\langle (B+iA)x,x \rangle|^2 = \langle Bx,x \rangle^2 +\langle Ax, x \rangle^2.
\end{align}
By the definition of $||\cdot||$, we compute
\begin{align}\label{7.36}
||(B+iA) x||^2 = ||B x||^2 + ||Ax||^2 - 2 \textrm{Im} \ \langle Ax, Bx \rangle.
\end{align}

Substituting (\ref{7.35}) and (\ref{7.36}) into (\ref{7.34}) yields the inequality,
\begin{equation}\label{7.37}
\begin{aligned}
||Ax||^2 - \langle & Ax, x \rangle ^2 +||Bx||^2 - \langle Bx, x \rangle ^2 \\
&\geq 2 \textrm{Im}\ \langle Ax, Bx \rangle.
\end{aligned}
\end{equation}

Letting $r,s \in \mathbb{R}$, so that $rA$ and $sB$ are also self-adjoint, (\ref{7.37}) becomes 
\begin{equation}\label{7.38}
\begin{aligned}
r^2\left(||Ax||^2 - \langle  Ax, x \rangle ^2 \right) & +s^2 \left( ||Bx||^2 - \langle Bx, x \rangle ^2\right) \\
&\geq 2 rs \textrm{Im}\ \langle Ax, Bx \rangle.
\end{aligned}
\end{equation}
Setting $r^2= ||Bx||^2 - \langle Bx, x \rangle ^2$ and $s^2= ||Ax||^2 - \langle Ax, x \rangle ^2$, 
substituting into (\ref{7.38}), squaring both sides and dividing, we obtain
\begin{align*}
\left( ||Ax||^2 - \langle Ax, x \rangle ^2 \right) \left(||Bx||^2 - \langle Bx, x \rangle ^2\right) \geq
 \left( \textrm{Im} \langle Ax, Bx \rangle \right)^2.
\end{align*}
From this inequality and (\ref{7.33}) the uncertainty principle inequality (\ref{eq:hspace1}) follows.
\end{proof}

In the same vein, and with the same dense domain of definition
constraints as in Theorem \ref{thm:hilbertup1}, we have --

\begin{thm}
\label{thm:hilbertup}
Let $A$ and $B$ be self-adjoint operators on a Hilbert space $H$. Define the 
self-adjoint operators $T = AB+BA$ (the anti-commutator) and $S = -i\,[A,B].$
Then, for a given state $x \in H,$ we have
\begin{equation}
\label{eq:hspace}
      \left( \inner{x}{Tx}^2 + \inner{x}{Sx}^2 \right)   \leq   4\langle A^2 \rangle \langle B^2 \rangle.
\end{equation}
Equality holds in \eqref{eq:hspace} if and only if there exists $z_0 \in \CC$ such that $Ax = z_0Bx$.

\begin{proof}
Applying the Cauchy-Schwarz inequality and self-adjointness of $A$ we obtain
\begin{equation}
\label{eq:hilbertup1}
     \langle A^2 \rangle \langle B^2 \rangle = \norm{Ax}^2 \norm{Bx}^2 \geq \left|\inner{Ax}{Bx} \right|^2 = 
     \left|\inner{x}{ABx}\right|^2.
\end{equation}
By definition of $T$ and $S$, we have $AB = \frac{1}{2}T + \frac{i}{2}S$. Therefore, 
\begin{equation}
\label{eq:hilbertup2}
        |\langle x , ABx \rangle |^2 = \frac{1}{4}\left|\inner{x}{(T+iS)x}\right|^2
\end{equation}
\[
        = \frac{1}{4}\left|\inner{x}{Tx} - i \inner{x}{Sx}\right|^2\nonumber \\
       = \frac{1}{4}\left(\inner{x}{Tx}^2 + \inner{x}{Sx}^2\right).
\]
The final equality holds because $\inner{x}{Tx}$ and $\inner{x}{Sx}$ are real, and \eqref{eq:hspace} 
follows from \eqref{eq:hilbertup1} and \eqref{eq:hilbertup2}. 

Last, equality holds if and only if 
we have equality in the application of Cauchy-Schwarz, and this occurs when $Ax$ and $Bx$ are 
linearly dependent.
\end{proof}
\end{thm}

\begin{example}
\label{ex:hilbertineqs}
The uncertainty principle inequalities (\ref{eq:hspace1}) and (\ref{eq:hspace}) can be compared
quantitatively by substituting the definitions of expected value and variance into the inequalities themselves.
As such, (\ref{eq:hspace1}) becomes
\[
       \left({\rm Im} \langle BA \rangle\right)^2 \; \leq \; \left(\langle A^2 \rangle  - \langle A \rangle^2\right) \,
       \left(\langle B^2 \rangle  - \langle B \rangle^2\right),
\]
and (\ref{eq:hspace}) becomes
\[
      \left({\rm Re} \langle BA \rangle\right)^2 - \left({\rm Im} \langle BA \rangle\right)^2 \; \leq 
      \; \langle A^2 \rangle \, \langle B^2 \rangle.
\]
\end{example}

Theorem \ref{thm:hilbertup} implies the more frequently used inequality for self-adjoint operators $A$ and $B,$
viz.,
\begin{equation}
 \label{eq:hspace2}
                       \left|\inner{[A,B]x}{x}\right|    \leq   2\norm{Ax}\norm{Bx}.
\end{equation}
Indeed, dropping the anti-commutator term from the right side of \eqref{eq:hspace} leaves
\[
         \inner{x}{Sx}^2 = \left|\inner{[A,B]x}{x}\right|^2.
\]
We have equality in \eqref{eq:hspace2} when $Ax$ and $Bx$ are linearly dependent (as above) and 
$\inner{x}{Tx} = 0$, i.e., when $\inner{Ax}{Bx}$ is completely imaginary. This weaker form of 
\eqref{eq:hspace} is enough to prove Theorem \ref{thm:heisenberg}, and thus the 
full content of Theorem \ref{thm:hilbertup} is 
usually neglected; however, we shall make use of it in Subsection \ref{sec:vvDFTuncertainty}.

Define the {\it position} and {\it momentum} operators respectively by
\[Qf(x) = xf(x), \quad Pf(x) = \frac{1}{2\pi i}f^\prime(x).\]
$Q$ and $P$ are densely defined linear operators on $L^2(\RR).$
When employing Hilbert space operator inequalities, such as \eqref{eq:hspace} and \eqref{eq:hspace2}, 
they are valid only for $x \in H$ in the domains of all the operators in question, i.e., $A$, $B$, $AB$, 
and $BA$.  We are now ready to prove Theorem \ref{thm:heisenberg} using the self-adjoint 
operator approach of this subsection, see \cite{BenDel2016} for other examples.

\begin{proof}[Proof of Theorem \ref{thm:heisenberg}]
Let $Q$ and $P$ be as defined above. Then, for $f, g \in D(Q)$, we have
\[
        \inner{Qf}{g} = \int xf(x)\overline{g(x)}\,dx = \int f(x)\overline{xg(x)}\,dx = \inner{f}{Qg},
\]
and for $f,g \in D(P)$,
\[
      \inner{Pf}{g} = \frac{1}{2\pi i} \int f^\prime(x)\overline{g(x)}\,dx = -\frac{1}{2\pi i} \int f(x)\overline{g^\prime(x)}\,dx 
      = \inner{f}{Pg}.
\]
Therefore $Q$ and $P$ are self-adjoint. The operators 
$Q-x_0$ and $P-\gamma_0$ are also self-adjoint and 
$[Q-x_0,P-\gamma_0] = [Q,P]$. Thus, \eqref{eq:hspace2} implies that for every $f$ in the domain of 
$Q$, $P$, $QP$, and $PQ$, e.g., $f$ a Schwartz function, 
\begin{equation}
\label{eq:heisproof1}
            \frac{1}{2}\left|\inner{[Q,P]f}{f}\right| \leq \norm{(Q-x_0)f}\norm{(P-\gamma_0)f}.
\end{equation}
For the commutator term we obtain
\begin{equation}
\label{eq:heisproof2}
           [Q,P]f(x) = \frac{1}{2\pi i}(xf^\prime(x)-(f^\prime(x)+xf^\prime(x))) = -\frac{1}{2\pi i} f(x).
\end{equation}
Combining \eqref{eq:heisproof1} and \eqref{eq:heisproof2} yields
\begin{equation*}
            \frac{1}{4\pi}\norm{f}_2^2 \leq \norm{(Q-x_0)f}\norm{(P-\gamma_0)f}.
\end{equation*}
It is an elementary fact from Fourier analysis that $(\frac{d}{dx}f\widehat{)}(\gamma) = 
2\pi i \gamma \widehat{f}(\gamma)$; applying this and Plancherel's theorem to the second term yields
\[
       \norm{(P-\gamma_0)f} = \left(\int (\gamma-\gamma_0)^2|\widehat{f}(\gamma)|^2\,d\gamma \right)^{1/2},
\]
and Heisenberg's inequality (\ref{eq:hup}) follows.
\end{proof}


\subsection{An uncertainty principle for the vector-valued DFT}
\label{sec:vvDFTuncertainty}

The uncertainty principle we prove for the vector-valued {\it DFT} is an 
extension of an uncertainty principle inequality proved by Gr\"{u}nbaum for the {\it DFT}
in \cite{grun2003}. We begin by defining two operators meant to represent the 
position and momentum operators defined on $\RR$ in Subsection \ref{sec:uncertainty}.

Define 
\begin{equation*}
         P: \ell^2(\ZZ/N\ZZ \times \ZZ/d\ZZ) \rightarrow \ell^2(\ZZ/N\ZZ \times \ZZ/d\ZZ)
\end{equation*}
by the formula,
\begin{equation}
\label{eq:P}
        \forall m \in \ZZ/N\ZZ, \quad P(u)(m) = i(u(m+1) - u(m - 1));
\end{equation}
and, given a fixed real valued $q \in \ell^2(\ZZ/N\ZZ \times \ZZ/d\ZZ)$, define 
\[
          Q: \ell^2(\ZZ/N\ZZ \times \ZZ/d\ZZ) \rightarrow \ell^2(\ZZ/N\ZZ \times \ZZ/d\ZZ)
\]          
by the formula
\begin{equation}
\label{eq:Q}
\forall m \in \ZZ/N\ZZ, \quad Q(u)(m) = q(m) u(m).
\end{equation}

\begin{prop}
The operators $P$ and $Q$ defined by (\ref{eq:P}) and  (\ref{eq:Q}) are linear and self-adjoint.
\begin{proof}
The linearity of $P$ and $Q$ and self-adjointness of $Q$ are clear. To show that $P$ is self-adjoint, 
let $u, v \in \ell^2(\ZZ/N\ZZ \times \ZZ/d\ZZ)$. We compute
\[
       \inner{Pu}{v} = \sum_{m=0}^{N-1} \inner{P(u)(m)}{v(m)}
= \sum_{m=0}^{N-1} \inner{i(u(m+1) - u(m-1))}{v(m)}
\]
\[
= \sum_{m=0}^{N-1} i \inner{u(m+1)}{v(m)} - i \inner{u(m-1)}{v(m)}\\
= \sum_{m=0}^{N-1} i \inner{u(m)}{v(m-1)} - i \inner{u(m)}{v(m+1)} 
\]
\[
= \sum_{m=0}^{N-1} \inner{u(m)}{i(v(m+1) - v(m-1))}\\
= \inner{u}{Pv}.
\qedhere\]
\end{proof}
\end{prop}

Define the anti-commutator $T = QP+PQ$ and $S = - i[Q,P]$. Because the  Hilbert space 
$H = \ell^2(\ZZ/N\ZZ \times \ZZ/d\ZZ)$ is finite dimensional, 
$T$ and $S$ are linear self-adjoint operators defined on all of $H.$ Applying Theorem \ref{thm:hilbertup} 
gives an uncertainty principle inequality for the operators $Q$ and $P$:
\begin{equation}
\label{eq:vvDFTuncertainty1}
         \forall u \in \ell^2(\ZZ/N\ZZ \times \ZZ/d\ZZ), \quad \left( \inner{u}{Tu}^2 + \inner{u}{Su}^2 \right)
         \leq 4 \langle Q^2 \rangle \langle P^2 \rangle. 
\end{equation}
In this form, (\ref{eq:vvDFTuncertainty1}) does not appear to be related to the vector-valued {\it DFT}. 
We shall make the connection by finding appropriate expressions for each of the terms 
in (\ref{eq:vvDFTuncertainty1}), thereby
yielding a form of the Heisenberg inequality for the vector-valued {\it DFT}.

The expected values of $Q$ and $P$ are
\[
\langle Q^2 \rangle = \inner{Qu}{Qu}\\
= \sum_{m=0}^{N-1} \inner{Q(u)(m)}{Q(u)(m)}
\]
\[
= \sum_{m=0}^{N-1} \inner{q(m)u(m)}{q(m)u(m)}\\
= \sum_{m=0}^{N-1} \norm{q(m)u(m)}_{\ell^2(\ZZ/d\ZZ)}^2
= \norm{q u}^2 
\]
and 
\[
        \langle P^2 \rangle = \inner{Pu}{Pu}
         = \norm{Pu}^2
                     = \norm{i(\tau_{-1}u - \tau_{1}u)}^2
\]
\[
              = \norm{\mathcal{F}(\tau_{-1}u - \tau_{1}u)}^2 
           = \norm{e^{1} \widehat{u} - e^{-1} \widehat{u}}^2 
          = \norm{(e^1 - e^{-1}) \widehat{u}}^2.
\]
In the computation of $ \langle P^2 \rangle$ we use the 
unitarity of the vector-valued {\it DFT} mapping $\Fcal$ and the fact that $e^1$ and $e^{-1}$ are 
the modulation functions $e^j(m) = x_{jm},$ for a given {\it DFT}
frame $\{x_k\}_{k=0}^{N-1}$ for 
$\mathbb C^d,$ see Definition \ref{defn:transmod}.

We restate these expected values:
\begin{equation}
\label{eq:QPnorms}
       \langle Q^2 \rangle = \norm{q u}^2 \text{ and} \quad
        \langle P^2 \rangle = \norm{(e^1 - e^{-1}) \widehat{u}}^2.
\end{equation}

We now seek expressions for the terms $\inner{u}{Tu}^2$ and $\inner{u}{Su}^2$. Computing the 
commutator and anti-commutator of $Q$ and $P$ gives 
\[
          i\,Su(m) = [Q,P]u(m) = i (q(m) - q(m+1))  u(m+1) - i(q(m) - q(m-1))  u(m-1)
\]
and
\[
           Tu(m) =(QP + PQ)u(m) = i(q(m) + q(m+1))  u(m+1) - i(q(m) + q(m-1))  u(m-1).
\]
Therefore,
\begin{equation}
\label{eq:anticommutator}
\inner{u}{Tu} =  \sum_{m = 0}^{N-1} \inner{u(m)}{T(u)(m)}
\end{equation}
\[
     = \sum_{m=0}^{N-1} \inner{u(m)}{i(q(m) + q(m+1))  u(m+1) - i(q(m) + q(m-1))  u(m-1)}\nonumber
\]
\[
   = i\sum_{m=0}^{N-1} \inner{u(m)}{(q(m) + q(m-1))  u(m-1)} - \inner{u(m)}{(q(m) + q(m+1))  u(m+1)}\nonumber\\
\]
\[
= i\sum_{m=0}^{N-1} \inner{(q(m) + q(m-1)) u(m)}{u(m-1)} - \inner{u(m)}{(q(m) + q(m+1))  u(m+1)}\nonumber\\
\]
\[
 = i\sum_{m=0}^{N-1} \inner{(q(m+1) + q(m)) u(m+1)}{u(m)} - \inner{u(m)}{(q(m) + q(m+1))  u(m+1)}\nonumber\\
 \]
 \[
= 2 \sum_{m=0}^{N-1} \text{Im}\inner{u(m)}{(q(m) + q(m+1))  u(m+1)},
\]
and
\begin{equation}
\label{eq:commutator}
        \inner{u}{Su} = \sum_{m=0}^{N-1} \inner{u(m)}{S(u)(m)}
\end{equation}
\[
= \sum_{m=0}^{N-1} \inner{u(m)}{(q(m) - q(m+1))  u(m+1) - (q(m) - q(m-1))  u(m-1)}\nonumber\\
\]
\[
 = \sum_{m=0}^{N-1} \inner{u(m)}{(q(m) - q(m+1))  u(m+1)} - \inner{u(m)}{(q(m) - q(m-1))  u(m-1)}\nonumber\\
 \]
 \[
    = \sum_{m=0}^{N-1} \inner{u(m)}{(q(m) - q(m+1))  u(m+1)} - \inner{(q(m+1) - q(m))  u(m+1)}{u(m)}\nonumber\\
\]
\[
= 2 \sum_{m=0}^{N-1} \text{Re} \inner{u(m)}{(q(m) - q(m+1)) u(m+1)}.
\]

Combining \eqref{eq:QPnorms}, \eqref{eq:anticommutator}, and \eqref{eq:commutator} 
with inequality (\ref{eq:vvDFTuncertainty1}) gives the
following general uncertainty principle for the vector-valued {\it DFT}.

\begin{thm}[General uncertainty principle for the vector-valued {\it DFT}]
\label{thm:genvvup}

\begin{equation}
\label{eq:genvvup}        
       \left(\sum_{m=0}^{N-1} {\rm Im}\inner{u(m)}{(q(m) + q(m+1))  u(m+1)}\right)^2
\end{equation}
\[
       + \left(\sum_{m=0}^{N-1} {\rm Re} \langle u(m), (q(m) - q(m+1)) u(m+1)\rangle \right)^2
       \leq  \norm{q  u}^2 \norm{(e^1 - e^{-1})  \widehat{u}}^2.
\]
\end{thm}

Theorem \ref{thm:genvvup} holds for any real valued $q$, but, to complete the analogy with
that of the classical uncertainty 
principle, we desire that the operators $Q$ and $P$ be unitarily equivalent through the Fourier transform, 
in this case, the vector-valued {\it DFT}. Indeed, setting $q =i(e^1 - e^{-1})$, we have 
$q(m)(n) = -2\sin(2\pi m s(n)/N)$ ($q$ is real-valued) and $\Fcal P = Q \Fcal$ as desired. With this choice 
of $Q$ we have proven the following version of the classical uncertainty principle for the vector-valued {\it DFT}.

\begin{thm}[Classical uncertainty principle for the vector-valued {\it DFT}]
\label{thm:classvvup}
Let $q = i(e^1 - e^{-1})$. 
For every $u$ in $\ell^2(\ZZ/N\ZZ \times \ZZ/d\ZZ)$ we have
\begin{equation}
\label{eq:classvvup}
  \left(\sum_{m=0}^{N-1} 
{\rm Im} \inner{u(m)}{(q(m) + q(m+1))  u(m+1)}\right)^2
\end{equation}
\[
     + \left(\sum_{m=0}^{N-1} {\rm Re} \inner{u(m)}{(q(m) - q(m+1)) u(m+1)}\right)^2
     \leq \norm{(e^1 - e^{-1})u}^2 \norm{(e^1 - e^{-1})\widehat{u}}^2.
\]
\end{thm}

\begin{rem}
\label{rem:taoetaup}
It is natural to extend the technique of Theorem \ref{thm:genvvup} to vector-valued
versions of recent uncertainty principle inequalities for finite frames  
\cite{LamMae2011}, graphs \cite{BenKop2015}, and cyclic groups and beyond \cite{tao-2005},
\cite{MurWha2012}.
\end{rem}



\section{Appendix: Unitary representations of locally compact groups}
\label{sec:unirep}

\subsection{Unitary representations}
Besides the references \cite{pont1966}, \cite{GelRaiShi1964}, \cite{rudi1962},  \cite{HewRos1963}, 
\cite{HewRos1970}, \cite{reit1968}, \cite{foll1995} cited 
in Subsection \ref{sec:vvBanachAlgebra}, fundamental and deep background
for this Appendix can also be found in \cite{sugi1975}, \cite{mack1978}, \cite{sund1987}.

Let $H$ be a Hilbert space over $\CC,$ and let $\Lcal(H)$ be the space of bounded linear
operators on $H.$ $\Lcal(H)$ is a $\ast-$Banach algebra with unit. In fact, one takes composition
of operators as multiplication, the
identity map $I$ is the unit, the operator norm gives the topology, and the involution $\ast$
is defined by the adjoint operator.

$\Ucal(H) \subseteq \Lcal(H)$ denotes the subalgebra of unitary operators $T$ on $H,$
i.e., $T\,T^{\ast} = T^{\ast}\,T = I.$

\begin{defn}[Unitary representation]
\label{def:representation}
Let $G$ be a locally compact group. A {\it unitary representation} of $G$ is a Hilbert space 
$H$ over $\CC$ and a homomorphism $\pi : G \rightarrow \Ucal(H)$ from $G$ into the group 
$\Ucal(H)$ of unitary operators on $H,$ that is continuous with respect to the strong operator 
topology on $\Ucal(H).$ (The strong operator topology is explicitly defined below. It is weaker
than the norm topology, and coincides with the weak operator topology on $\Ucal(H).$)
We spell-out these properties here for convenience:
\benm
\item $\forall g,h \in G,$ $\pi(gh) = \pi(g)\pi(h);$
\item $\forall g \in G,$ $\pi(g^{-1}) = \pi(g)^{-1} = \pi(g)^\ast,$ where $\pi(g)^\ast$ is the adjoint of $\pi(g);$
\item $\forall x \in H,$ the mapping $G \rightarrow H,$ $g \mapsto \pi(g)(x),$
is continuous.
\ennm
The dimension of $H$ is called the {\it dimension} of $\pi$. When $G$ is a finite group, then 
$G$ is given the discrete topology and the continuity of $\pi$ is immediate. 
We denote a representation by $(H, \pi)$ or, when $H$ is understood by $\pi$. 
\end{defn}

\begin{defn}[Equivalence of representations]
Let $(H_1,\pi_1)$ and $(H_2, \pi_2)$ be representations of $G$. A bounded linear map 
$T: H_1 \rightarrow H_2$ is an {\it intertwining operator} for $\pi_1$ and $\pi_2$ if
\[\forall g \in G, \quad T\pi_1(g) = \pi_2(g)T.\]
$\pi_1$ and $\pi_2$ are said to be {\it unitarily equivalent} if there is a unitary intertwining operator $U$ 
for $\pi_1$ and $\pi_2$. 
\end{defn}

More generally, we could consider non-unitary representations, where $\pi$ is a homomorphism into the 
space of invertible operators on a Hilbert space. We do not do that here for two reasons. First, we are
mainly interested in the regular representations (see Example \ref{ex:trans}) and
these are unitary, and, second, every finite dimensional 
representation of a finite group is unitarizable. That is, if $(H, \pi)$ is a finite dimensional representation 
(not necessarily unitary) of $G$ and $|G|<\infty$, then there exists an inner product on $H$ such that 
$\pi$ is unitary. See Theorem 1.5 of \cite{kosm2009} for a proof of this fact. 

\begin{example}
\label{ex:trans}
 Let $G$ be a finite group, and let $\ell^2=\ell^2(G)$. The action of $G$ on $\ell^2$ by left translation 
 is a unitary representation of $G$. More concretely, let $\{x_h\}_{h \in G}$ be the standard orthonormal basis for $\ell^2$, and define $\lambda : G \rightarrow \Ucal(\ell^2)$ by the formula,
\[
      \forall g, h \in G, \quad \lambda(g)x_h = x_{gh}.
\]
$\lambda$ is called the {\it left regular representation} of $G$. The {\it right regular representation}, which 
we denote by $\rho$, is defined as translation on the right, i.e.,
\[\forall g, h \in G, \quad \rho(g)x_h = x_{hg^{-1}}.\] 
The construction is similar for general locally compact groups and takes place on $L^2(G)$.
\end{example}


\subsection{Irreducible Representations}
\begin{defn}[Invariant subspace]
An {\it invariant subspace} of a unitary representation $(H,\pi)$ is a closed subspace $S \subseteq H$ 
such that $\pi(g)S \subseteq S$ for all $g \in G$. The restriction of $\pi$ to $S$ is a unitary representation 
of $G$ called a {\it subrepresentation}. If $\pi$ has a nontrivial subrepresentation, i.e., nonzero and not 
equal to $\pi$, or equivalently, if it has a nontrivial invariant subspace, then $\pi$ is {\it reducible}. If $\pi$ 
has no nontrivial subrepresentations or, equivalently, has no nontrivial invariant subspaces, then $\pi$ is 
{\it irreducible}. 
\end{defn}

\begin{defn}[Direct sum of representations]
Let $(H_1,\pi_1)$ and $(H_2, \pi_2)$ be representations of $G$. Then,
\[(H_1 \oplus H_2, \pi_1\oplus\pi_2),\]
where $(\pi_1\oplus\pi_2)(g)(x_1,x_2) = (\pi_1(g)(x_1),\pi_2(g)(x_2)),$ for $g \in G, x_1 \in H_1, x_2 \in H_2,$ 
is a representation of $G$ called the {\it direct sum} of the representations 
$(H_1,\pi_1)$ and $(H_2, \pi_2)$.
\end{defn}

More generally, for a positive integer $m$, we recursively define the direct sum of $m$ representations 
$\pi_1\oplus\ldots\oplus\pi_m$. If $(H, \pi)$ is a representation of $G$, we denote by $m\pi$ the 
representation that is the product of $m$ copies of $\pi$, i.e., 
$$
(H\oplus\ldots\oplus H, \pi\oplus\ldots\oplus\pi),
$$
where each sum has $m$ terms. Clearly, a direct sum of nontrivial representations cannot be irreducible, 
e.g., $(H_1\oplus H_2,\pi_1\oplus\pi_2)$ will have invariant subspaces $H_1 \oplus \{0\}$ and $\{0\}\oplus H_2$. 
\begin{defn}[Complete reducibility]
A representation $(H, \pi)$ is called {\it completely reducible} if it is the direct sum of irreducible representations.
\end{defn}
Two classical problems of harmonic analysis on a locally compact group $G$ are to describe all the 
unitary representations of $G$ and to describe how unitary representations can be built as direct 
sums of smaller representations. For finite groups, Maschke's theorem,
Theorem \ref{thm:maschkes}, tells us that 
the irreducible representations are the building blocks of representation theory that enable these descriptions.

\begin{lem}\label{lem:invariant}
Let $(H,\pi)$ be a unitary representation of $G$. If $S \subseteq H$ is invariant under $\pi$, 
then $S^\perp = \{y \in H: \forall x \in S, \inner{x}{y} = 0\}$ is also invariant under $\pi$.
\begin{proof}
Let $y \in S^\perp$. Then, for any $x \in S$ and $g \in G$, we have $\inner{x}{\pi(g)y} = 
\inner{\pi(g^{-1})x}{y} = 0$; and,
therefore, $\pi(g)y \in S^\perp$.
\end{proof}
\end{lem}

\begin{thm}[Maschke's theorem]
\label{thm:maschkes}
Every finite dimensional unitary representation of a finite group $G$ is completely reducible.
\begin{proof}
Let $(H, \pi)$ be a representation of a finite group $G$ with dimension $n < \infty$. If $\pi$ is irreducible, then
we are done. Otherwise, let $S_1$ be a nontrivial invariant subspace of $\pi$. By Lemma \ref{lem:invariant}, 
$S_2 = S_1^\perp$ is also an invariant subspace of $\pi$. Letting $\pi_1$ and $\pi_2$ be the restrictions of 
$\pi$ to $S_1$ and $S_2$ respectively, we have $\pi = \pi_1 \oplus \pi_2$, $\dim S_1 < n$, and $\dim S_2 < n$. 
Proceeding inductively, we obtain a sequence of representations of strictly decreasing dimension, which must 
terminate and yield a decomposition of $\pi$ into a direct sum of irreducible representations.
\end{proof}
\end{thm}

If $(H, \pi)$ is a unitary representation, we let $\Ccal_\pi \subseteq \Lcal(H)$ denote the algebra 
of operators on $H$ such that
\[
        \forall g \in G \; {\rm and} \; \forall T\in  \Ccal_\pi, \quad  T\,\pi(x) = \pi(x)\,T.
\]
$\Ccal_\pi$ is closed under taking weak limits and under taking adjoints, and, hence, it is a von Neumann
algebra. $\Ccal_\pi$ is the {\it commutant} of $\pi,$ and it is generated by
$\{\pi(g)\}_{g \in G}$. If $G$ is a finite group, then 
\[
       \Ccal_\pi = \left\{\sum_g a_g \pi(g): \{a_g\}_{g \in G} \subseteq \CC\right\}.
\]
Schur's lemma describes the commutants of irreducible unitary representations.

\begin{lem}[Schur's lemma, e.g., Lemma 3.5 of \cite{foll1995}] 
\label{lem:schurs}
Let $G$ be a locally compact group.\
\begin{enumerate}
\item Let $(H, \pi)$ be a unitary representation  of $G.$ $(H, \pi)$ is irreducible if and only if 
$\Ccal_\pi$ contains only scalar 
multiples of the identity.
\item Assume $T$ is an intertwining operator for irreducible unitary representations $(H_1, \pi_1)$ and 
$(H_2, \pi_2)$ of $G$. If $\pi_1$ and $\pi_2$ are inequivalent, then $T = 0$. 
\item If $G$ is Abelian, then every irreducible unitary representation of $G$ is one-dimensional.
\end{enumerate}
\end{lem}


\bibliographystyle{amsplain}
\bibliography{2017-03-14JBbib}

\end{document}